\documentclass[]{amsart}

\usepackage{graphicx,color}
\usepackage{amssymb}
\usepackage[initials,nobysame]{amsrefs}

\makeatletter
\@addtoreset{equation}{section}
\def\theequation{\thesection.\arabic{equation}}
\makeatother

\newtheorem{theorem}{Theorem}[section]
\newtheorem{proposition}[theorem]{Proposition}

\newtheorem{lemma}[theorem]{Lemma}
\newtheorem{problem}[theorem]{Problem}
\theoremstyle{definition}
\newtheorem{definition}[theorem]{Definition}
\newtheorem{remark}[theorem]{Remark}
\newtheorem{example}[theorem]{Example}

\newtheorem*{acknowledgments}{Acknowledgments}

\newcommand{\R}{\mathbf{R}}
\DeclareMathOperator{\rad}{rad}
\DeclareMathOperator{\conv}{conv}
\DeclareMathOperator{\bd}{bd}
\newcommand{\bo}{o}

\begin{document}

\title[DC structure of the singular set of distance functions]{Delta-convex structure of the singular set of distance functions}
\author[T.~Miura]{Tatsuya Miura}
\address[T.~Miura]{Department of Mathematics, Tokyo Institute of Technology, Meguro, Tokyo 152-8511, Japan}
\email{miura@math.titech.ac.jp}
\author[M.~Tanaka]{Minoru Tanaka}
\address[M.~Tanaka]{School of Science, Department of Mathematics, Tokai University, Hiratsuka, Kanagawa 259-1292, Japan}
\email{tanaka@tokai-u.jp}

\date{\today}

\begin{abstract}
  For the distance function from any closed subset of any complete Finsler manifold, we prove that the singular set is equal to a countable union of delta-convex hypersurfaces up to an exceptional set of codimension two.
  In addition, in dimension two, the whole singular set is equal to a countable union of delta-convex Jordan arcs up to isolated points.
  These results are new even in the standard Euclidean space and shown to be optimal in view of regularity.
\end{abstract}

\maketitle

\setcounter{tocdepth}{1}
\tableofcontents

\section{Introduction}\label{section:1}

Let $m\geq2$ and $M^m$ be an $m$-dimensional connected complete smooth Finsler manifold throughout this paper, if not specified.
For a (nonempty) closed subset $N\subset M$ we consider the distance function $d_N:M\setminus N\to(0,\infty)$ from $N$:
\begin{equation}\label{eq:distance_def}
  d_N(p):=d(N,p)=\inf_{q\in N}d(q,p).
\end{equation}
We define the \emph{singular set} $\Sigma(N)$ by
\begin{equation}\label{eq:def_singular_set}
    \Sigma(N):=\Sigma(d_N),
\end{equation}
where for a general function $f:\Omega\subset M\to\R$ the singular set $\Sigma(f)$ is defined by
\begin{equation*}
    \Sigma(f):=\{p\in\Omega \mid \textrm{$f$ is not differentiable at $p$}\}.
\end{equation*}

The distance function emerges in a wide range of fields in geometry, analysis, and applied mathematics, not only as a fundamental tool but also a research object in itself.
For instance, Euclidean (resp.\ Finslerian) distance functions directly appear as natural weak solutions to the eikonal equation (resp.\ Hamilton--Jacobi equations).
In fact, even in Euclidean spaces, Finslerian distance functions directly appear as viscosity solutions $u:\overline{\Omega}\to\R$ to first-order Hamilton--Jacobi equations of the form
\begin{align*}
  \begin{cases}
    \begin{array}{rl}
      H(x,Du) = 1 & \quad \mathrm{in}\ \Omega\subset\mathbf{R}^m, \\
      u = 0 &  \quad \mathrm{on}\ \partial\Omega,
    \end{array}
  \end{cases}
\end{align*}
where $H\in C^\infty(\mathbf{R}^m\times\mathbf{R}^m)$ satisfies certain structural assumptions including the uniform convexity of $V_p:=\{H(p,\cdot)<1\}$;
see e.g.\ the classical book of P.-L.\ Lions \cite{Lions1982} and the celebrated study of Li--Nirenberg \cite{Li2005} for details.
A central problem in analysis of Hamilton--Jacobi equations is to understand the structure of the singular set of solutions.
From a geometrical point of view, the set $\Sigma(N)$ is equivalent to the \emph{ambiguous locus}, which is closely related to the cut locus.
Since the cut locus often appears as an important obstruction, its precise structure has been and is still investigated extensively.
In addition, the ambiguous locus is also called \emph{medial axis} \cite{Blum1967} (cf.\ \cite{Farin2002,Lieutier2004,Attali2009}) or \emph{skeleton} (cf.\ \cite{Duda1973}) in applied fields involving shape recognition, used as a fundamental tool as far back as the classical notion of Voronoi diagram or Dirichlet tessellation (where $N$ is discrete).
In fact, for a bounded open set $\Omega$ in a Riemannian manifold $M$ the set $\Sigma(M\setminus\Omega)$ has the same homotopy type as $\Omega$ \cite{Lieutier2004,Albano2013}.

The purpose of this paper is to investigate the general structure and optimal regularity of $\Sigma(N)$ under no more assumption on $N$.
The generality of $N$ would be important in view of the aforementioned wide applications, while it is known that the irregularity of $N$ may cause some pathology; in fact, concerning the closure of the singular set, $\overline{\Sigma(K)}=\R^m$  \cite{Zamfirescu1990} holds for most compact sets $K\subset\R^m$, and even $\overline{\Sigma(\partial C)}=C$ holds for most convex bodies $C$ with $C^1$-boundary $\partial C$ \cite{Santilli2021} (cf.\ \cite{Miura2016}).
However there is still a chance to obtain a finer structure theorem for $\Sigma(N)$ itself.

\subsection{Main results}

For a general closed subset $N\subset M$ it is well known that $\Sigma(N)$ is rectifiable, and in fact a stronger structural upper bound follows by Zaj\'{i}\v{c}ek's result below.
However its lower-bound counterpart is much more delicate, and our main contribution will be devoted to this point.

We shall first recall a remarkable characterization of the singular set of a general convex function in terms of delta-convex submanifolds, due to Zaj\'{i}\v{c}ek \cite{Zajicek1979} (necessity) and Pavlica \cite{Pavlica2004} (sufficiency).
A function is said to be \emph{delta-convex}, or \emph{DC} for short, if it is represented by a difference of two convex functions.
In particular, in view of regularity, DC is stronger than $C^{0,1}$ (Lipschitz) but weaker than $C^{1,1}$.
A submanifold is DC if it is locally represented by a DC graph.
See Section \ref{sec:StructureEuclidean} for details.
In terms of delta-convexity, they obtained the following

\begin{theorem}[Zaj\'{i}\v{c}ek \cite{Zajicek1979},  Pavlica \cite{Pavlica2004}]\label{thm:Zajicek-Pavlica}
  If $\Sigma(f)$ denotes the singular set of a convex function $f:\R^m\to\R$, then $\Sigma(f)$ is an $F_\sigma$-subset of a countable union of delta-convex hypersurfaces in $\R^m$.
  Conversely, if $A\subset\R^m$ is an $F_\sigma$-subset of a countable union of delta-convex hypersurfaces in $\R^m$, then there exists a convex function $f:\R^m\to\R$ such that $\Sigma(f)=A$.
\end{theorem}

The distance function $d_N$ is locally semi-concave (cf.\ \cite{Cannarsa2004}), i.e., locally concave up to an addition of a parabola, and hence the local structural properties of the singular set $\Sigma(N)$ directly inherit from those of convex functions.
In particular, extending the ambient space from $\R^m$ to $M^m$ (cf.\ Section \ref{sec:StructureFinslerian}), we deduce from Theorem \ref{thm:Zajicek-Pavlica} the structural upper bound that \emph{the singular set $\Sigma(N)$ is covered by a countable union of DC hypersurfaces $S_j\subset M^m$, that is, $\Sigma(N)\subset\bigcup_{j=1}^\infty S_j$.}
This assertion is clearly stronger than the standard countable $(m-1)$-rectifiability, or even than the rectifiability of class $C^2$ thanks to a covering argument (cf.\ \cite{Hajlasz2022} and references therein).

A natural problem is then to explore the structure of the singular set itself, by obtaining suitable structural lower bounds.
In view of Theorem \ref{thm:Zajicek-Pavlica}, the set $\Sigma(N)=\Sigma(d_N)$ is at least an $F_\sigma$-subset, but this topological condition is not strong enough for ruling out many pathological examples such as Cantor sets.

Our main result ensures that the singular set is in fact \emph{equal to} a union of DC hypersurfaces up to an exceptional set of codimension two.

\begin{theorem}[Generic structure in any dimension]\label{thm:main_singular_set}
  Let $m\geq2$ and $N\subset M^m$ be a closed subset of an $m$-dimensional connected complete Finsler manifold.
  Then
  $$\Sigma(N)=\bigcup_{j=1}^\infty S_j\cup R,$$
  where $\{S_j\}_{j=1}^\infty$ is a countable family of delta-convex hypersurfaces $S_j\subset M$, and $R\subset M$ is a subset contained in a countable union of $(m-2)$-dimensional delta-convex submanifolds of $M$.
\end{theorem}

\begin{remark}[Optimality]
    It turns out that the DC regularity is optimal.
    Indeed, in Proposition \ref{prop:optimality} we will construct an example which shows that the DC regularity cannot be replaced with a slightly better one such as $C^1$ or semi-concavity.
\end{remark}

Our next theorem gives a finer understanding of the whole structure in dimension two.
If $m=2$ in Theorem \ref{thm:main_singular_set}, then the set $R$ consists of at most countably many points.
However, this observation does not tell us anything about the relation between zero- and one-dimensional parts.
Here we show that those parts are completely ``splitted'' in the sense that the whole singular set is equal to a union of DC Jordan arcs (corresponding to one-dimensional DC submanifolds ``with boundary'', cf.\ Definition \ref{def:DC_Jordan_arc}) and isolated points.

\begin{theorem}[General structure in dimension two]\label{thm:2D_singular_set}
  Let $N\subset M^2$ be a closed subset of a two-dimensional connected complete Finsler manifold.
  Then
  $$\Sigma(N)=\bigcup_{j=1}^\infty S_j\cup R,$$
  where $\{S_j\}_{j=1}^\infty$ is a countable family of delta-convex Jordan arcs $S_j\subset M$, and $R\subset M$ consists of all isolated points of $\Sigma(N)$.
  Clearly, $R\cap S_j=\emptyset$ for all $j$.
\end{theorem}


The significance of Theorem \ref{thm:main_singular_set} or Theorem \ref{thm:2D_singular_set}, compared to the known structural upper bound from Theorem \ref{thm:Zajicek-Pavlica}, is highlighted by Cantor-like sets for example.
In fact, a concave function can have the singular set given by a Cantor-like set $C\subset\R\times\{0\}\subset\R^2$ of any Hausdorff dimension between $0$ and $1$.
The set $C$ is $1$-rectifiable as it can be covered by a segment.
However $C$ cannot be attained by any singular set $\Sigma(N)\subset\R^2$ since Theorem \ref{thm:2D_singular_set} with $M=\R^2$ implies that the set $\Sigma(N)$ must have Hausdorff dimension $0$ or $1$.
(See Section \ref{subsec:counterexample} for more details.)

Our argument for Theorem \ref{thm:main_singular_set} also extends to a proof of higher-codimensional singularity propagations (Theorem \ref{thm:high_biLipschitz}), strengthening the assertion of \cite[Corollary 6.4]{Albano1999}, but only works for non-generic points (cf.\ Problem \ref{prob:highercodimension}).
In particular the following problem would be open, even in the Euclidean case $M^m=\R^m$ and even if DC regularity is weakened to Lipschitz regularity:

\begin{problem}\label{prob:DCcovering}
  For $m\geq3$ and a closed subset $N\subset M^m$, is the singular set $\Sigma(N)$ equal to a countable union of DC submanifolds?
\end{problem}

\subsection{Literature}

Now we compare our study with the previous literature.
In particular, Theorem \ref{thm:main_singular_set} is new even for the standard Euclidean space $M=\mathbf{R}^m$.
Roughly speaking, previously known are a ``Lipschitz propagation'' along the singular set of a semi-concave function on $\R^m$ \cite{Albano1999}, and a ``bi-Lipschitz propagation'' for the distance function on $M^2$ \cite{Sabau2016}.
Compared to those results, our finer assertion may be called ``DC graphical propagation''.

The analytical study of the singular set of distance functions has a long history, a part of which we now briefly review focusing on the standard Euclidean case $M=\R^m$.

The primitive upper bound that $\Sigma(N)\subset\mathbf{R}^m$ has zero Lebesgue measure follows by Rademacher's theorem, but it was also geometrically shown by Erd\H{o}s in 1945 \cite{Erdoes1945}.
In 1946 \cite{Erdoes1946}, Erd\H{o}s also proved that $\Sigma(N)$ is covered by a countable union of sets of finite $(m-1)$-dimensional Hausdorff measure, and more on higher-codimensional structures.
Finer structural upper bounds are by now well developed for general concave functions, and thus for semi-concave functions including distance functions.
In fact, Zaj\'{i}\v{c}ek \cite{Zajicek1979} studied the $k$-th singular set $\Sigma^k(u):=\{x\in\mathbf{R}^m\mid \dim(\partial u(x))\geq k \}$ for $k\geq1$ of a general concave function $u:\mathbf{R}^m\to\mathbf{R}$, where $\partial u$ denotes the supergradient (or the generalized gradient), and proved (not only Theorem \ref{thm:Zajicek-Pavlica} but also) that $\Sigma^k(u)$ is covered by a countable union of $(m-k)$-dimensional DC submanifolds.
In particular, $\dim_{\mathcal{H}}\Sigma^k(u)\leq m-k$, where $\dim_\mathcal{H}$ denotes the Hausdorff dimension.
Zaj\'{i}\v{c}ek's DC result also implies that $\Sigma^k(u)$ is $C^2$-rectifiable; this fact is later independently found by Alberti \cite{Alberti1994} and extended by Menne--Santilli \cite{Menne2019}.
Alberti--Ambrosio--Cannarsa \cite{Alberti1992} obtained a quantitative upper bound for $\Sigma^k(u)$.

On the other hand, the structural lower bound is more delicate.
This problem is often called \emph{propagation of singularities} in the context of semi-concave functions or solutions to Hamilton--Jacobi equations (see e.g.\ the surveys \cite{Cannarsa2020,Cannarsa2021} and references therein).
Apart from many important contributions to 1D dynamical propagations (as in \cite{Cannarsa2020,Cannarsa2021}), a few results are available for higher-dimensional propagations.
Ambrosio--Cannarsa--Soner \cite{Ambrosio1993} addressed this issue for general semi-concave functions and gave conditions for singularities to have lower bounds of the Hausdorff dimension.
Albano--Cannarsa proceeded in this direction and obtained a remarkable ``Lipschitz propagation'' result \cite[Theorem 5.2]{Albano1999}: Under a geometrical condition on the generalized gradient, cf.\ \eqref{eq:AlbanoCannarsa}, singularity propagates along the image of a nontrivial Lipschitz map; see also \cite[Theorem 4.2]{Cannarsa2009} for more on how it propagates.
In particular, by applying the results in \cite{Alberti1992,Albano1999} to the distance function $d_N$, one can assert that such a Lipschitz propagation occurs at a generic singular point (cf.\ Section \ref{subsec:semiconcave}).
This result already implies that either $\dim_\mathcal{H}\Sigma(N)=m-1$ or $\dim_\mathcal{H}\Sigma(N)\leq m-2$ holds, already ruling out the above Cantor-like examples.
Compared to Albano--Cannarsa's result, Theorem \ref{thm:main_singular_set} asserts finer properties; the propagation is graphical and has optimal DC regularity, and the covering is countable.
Theorem \ref{thm:2D_singular_set} improves Bartke--Berens' result
\cite{Bartke1986} to optimal DC regularity (see also \cite{Frerking1989,Vesely1992,Albano1999} for related arc-type propagation results).
We also emphasize that all the aforementioned previous results are obtained in Euclidean space but our results are valid in a general Finsler manifold.

We mention that the understanding of delta-convexity in terms of integrability is very recently gained by Lions--Seeger--Souganidis and Tao \cite[Appendix B]{LionsIUMJ}.
In general we have $W^{2,\infty}_\mathrm{loc}(\R^d)\subset DC_\mathrm{loc}(\R^d)\subset W^{1,\infty}_\mathrm{loc}(\R^d)$ as mentioned.
If $d=1$, we have $u\in DC_\mathrm{loc}(\R)$ if and only if $u'\in BV_\mathrm{loc}(\R)$.
On the other hand, if $d\geq2$, Tao's counterexample shows that $W^{2,p}_\mathrm{loc}(\R^d)\setminus DC_\mathrm{loc}(\R^d)\neq\emptyset$ for all $p\in[1,\infty)$.

We also discuss some geometric aspects based on the cut locus.
The cut locus from $N$ is the set of all points where a geodesic from a closed set $N$ loses its minimality.
In classical settings, e.g.\ for smooth submanifolds $N$, it is well known that the cut locus is decomposed into two parts; the set called ambiguous locus, which consists of points admitting multiple minimal geodesics to $N$, and the remaining set of points admitting unique minimal geodesics (which are necessarily focal points).
The ambiguous locus is known to be dense in the cut locus and in fact characterized by the singular set $\Sigma(N)$, cf.\ \cite{Sabau2016}.

The cut locus was first introduced by Poincar\'{e} \cite{Poincare1905} in 1905 in the case of a singleton $N=\{p\}$ in a certain two-dimensional Riemannian manifold $M$, where it is shown that the cut locus is a union of arcs with finitely many ``endpoints''.
Since then similar or more precise structural results have been established by many authors for the cut locus from a point in various two-dimensional ambient spaces, see e.g.\ \cite{Myers1935,Myers1936,Whitehead1935,Hartman1964,Hebda1994,Itoh1996}.
The case of a general closed set $N\subset M^m$ is also well understood for dimension $m=2$, namely for Alexandrov surfaces \cite{Shiohama1996} or Finsler surfaces \cite{Sabau2016} (see also \cite{Tanaka2020} dealing with a wider class of functions).
In particular, in \cite{Sabau2016}, a very similar result to Theorem \ref{thm:main_singular_set} is obtained in dimension two; namely, a generic part of the singular set is equal to a countable family of rectifiable Jordan arcs.
Theorem \ref{thm:main_singular_set} extends this result to a general dimension by a different method, and also Theorem \ref{thm:2D_singular_set} gives a stronger (likely optimal) assertion in dimension two.

Concerning higher-dimensional ambient spaces $M$, one of the most relevant results would be Hebda's statement \cite[Proposition 1.1]{Hebda1987}, which is strongly based on Ozols' results on hypersurface properties \cite[Propositions 2.3, 2.4]{Ozols1974}: Let $N=\{p\}$ in a complete Riemannian manifold $M^m$.
Let $C_2(N)$ be the set of all cut points that are nonconjugate and admit exactly two minimal segment from $N=\{p\}$ (called \emph{normal cut points} by Hartman \cite{Hartman1964}, and also \emph{cleave points} by Hebda \cite{Hebda1987}).
Then $C_2(N)$ is relatively open in the cut locus and consists of smooth hypersurfaces, and also the remaining set has zero $\mathcal{H}^{m-1}$-measure.
In this regard, our Theorem \ref{thm:main_singular_set} extends those results with respect to the generality of $N$ (as well as $M$).
The fact that $N$ is not a singleton yields many substantial difficulties.
For example, due to the possible irregularity of $N$ one has no canonical notion of (non)conjugacy of cut points.
Here we simply do not deal with conjugacy but consider the set $\Sigma_2(N)$ of all points admitting two minimal geodesics from $N$.
Then the set $\Sigma_2(N)$ may not be relatively open in general, which is a new technically delicate point.

For the proof of Theorem \ref{thm:main_singular_set} we develop a proper extension of an implicit function method classically used in differential geometry, e.g.\ in \cite{Hartman1964,Ozols1974,Hebda1987,Ardoy}.
More precisely we focus on the generic subset $\Sigma_2(N)\subset\Sigma(N)$ and prove that a DC graphical propagation occurs from any point of $\Sigma_2(N)$ by applying an implicit function theorem of Zaj\'{i}\v{c}ek.
The countability of the covering is an additional delicate point because the $\mathcal{H}^{m-1}$-measure of $\Sigma(N)$ may not be locally finite, and also because the propagation from a point in $\Sigma_2(N)$ may not cleanly ``cleave'' the singular set, i.e., even after propagation there may remain non-negligible residual points in a neighborhood (cf.\ an example in \cite{Mantegazza2003}, or our Example \ref{ex:branch}).
To overcome this issue we introduce a new quantitative level of hierarchy and prove a quantitative ``cleaving'' property.
The complement $\Sigma(N)\setminus\Sigma_2(N)$ is estimated by Zaj\'{i}\v{c}ek's upper bound.
For the proof of Theorem \ref{thm:2D_singular_set} we develop an additional sector-decomposition argument inspired by \cite{Sabau2016}.

Finally we mention that the closure of $\Sigma(N)$, an analytical counterpart of the cut locus, is also called \emph{ridge} and importantly appears in various contexts; e.g., elastic-plastic torsion problems \cite{Ting1966,Caffarelli1979}, granular matter theory \cite{Crasta2015}, infinity-Laplacian eigenvalue problem \cite{Crasta2019}.
General theory is developed \cite{Mantegazza2003,Li2005,Crasta2007,Miura2016,Santilli2021,Bialozyt} and it is revealed that the wildness of the ridge is quite sensitively affected by the irregularity of $N$.
Other types of singular sets and their relations are also under investigation by many authors, see e.g.\ \cite{Fremlin1997,Bishop2008,Crasta2016,Safdari2019,Kolasinski}.

\subsection{Organization}

This paper is organized as follows:
To clarify the main ideas for the generic structure theorem, in Section \ref{sec:StructureEuclidean}, we first prove Theorem \ref{thm:main_singular_set} in the Euclidean case $M=\mathbf{R}^m$ (Theorem \ref{thm:Euclidean_singular_set}) and discuss more, including optimality of DC regularity (Proposition \ref{prop:optimality}) and comparison with semi-concave function theory.
We then turn to the Finsler case.
In Section \ref{sec:PreliminariesFinsler} we recall some basic notions and prepare basic definitions in terms of Finsler geometry, and then in Section \ref{sec:StructureFinslerian} we complete the proof of Theorem \ref{thm:main_singular_set} through a more precise statement (Theorem \ref{thm:Finslerian_singular_set}).
Finally we prove Theorem \ref{thm:2D_singular_set} in Section \ref{sec:Structure2D}.

\begin{acknowledgments}
  The authors would like to thank Luis Guijarro and Ulrich Menne for their helpful comments.
  The first author is in part supported by JSPS KAKENHI Grant Numbers 18H03670, 20K14341, and 21H00990, and by Grant for Basic Science Research Projects from The Sumitomo Foundation.
\end{acknowledgments}

\section{Generic DC structure: Euclidean case}\label{sec:StructureEuclidean}

In this section we prove Theorem \ref{thm:main_singular_set} in the Euclidean case $M=\mathbf{R}^m$ (with the standard metric) for clarifying the essential techniques, and also for convenience of some of the readers unfamiliar with non-Euclidean arguments.

\subsection{Lipschitz, delta-convex, semi-concave, and distance functions}

In this preliminary section we first clarify the relationship between some important classes of functions.
Let $\Omega\subset\R^m$ be an open set and $u:\Omega\to\R$ be a continuous function.
\begin{itemize}
    \item We say that $u$ is \emph{locally Lipschitz} if for any open set $\Omega'$ compactly embedded in $\Omega$ there is $C>0$ such that $|u(x)-u(y)|\leq C|x-y|$ holds for any $x,y\in\Omega'$.
    \item We say that $u$ is \emph{locally delta-convex}, or \emph{locally DC} for short, if for any open convex set $\Omega'$ compactly embedded in $\Omega$ there are two convex functions $g_1,g_2:\Omega'\to\R$ such that $u=g_1-g_2$ holds on $\Omega'$.
    \item We say that $u$ is \emph{locally semi-concave} if for any open convex set $\Omega'$ compactly embedded in $\Omega$ there is $C>0$ such that one of the following equivalent conditions hold (see Cannarsa--Sinestrari \cite{Cannarsa2004} for details):
    \begin{itemize}
      \item[(1)] $\lambda u(x_1)+(1-\lambda)u(x_0) - u(\lambda x_1+(1-\lambda)x_0) \leq C\lambda(1-\lambda)|x_1-x_0|^2$ for any $\lambda\in[0,1]$ and any $x_0,x_1\in \Omega'$.
      \item[(2)] $u(x)-C|x|^2$ is a concave function on $\Omega'$.
      \item[(3)] $D^2u\leq 2CI$ on $\Omega'$ in the distributional sense, where $I$ denotes the $m\times m$ identity matrix.
    \end{itemize}
\end{itemize}
We also say that a map $u:\Omega\to\R^k$ is locally Lipschitz (resp.\ DC, semi-concave) map if so is every component of $u=(u_1,\dots,u_k)$.

Any locally semi-concave function is locally delta-convex by property (2).
Also any locally delta-convex function is locally Lipschitz since so is any convex function.
In addition, it is well known \cite{Cannarsa2004} that
\begin{itemize}
    \item For any closed subset $N\subset\R^m$ the distance function $d_N:\R^m\setminus N \to \R$ is locally semi-concave (and hence DC as well as Lipschitz).
\end{itemize}
We will use these notions and facts throughout this paper.
\emph{We will often drop the term ``locally''} because our concern will be completely local.

For later use we collect fundamental properties of DC functions, which can be found e.g.\ in Hartman \cite{Hartman1959}, Vesel\'{y}--Zaj\'{i}\v{c}ek \cite{Vesely1989}, Lions--Seeger--Souganidis \cite{LionsIUMJ}.

\begin{lemma}\label{lem:DC_fundamental_properties}
Let $U$ and $V$ denote open subsets of $\R^m$ and $\R^k$ respectively.
\begin{itemize}
    \item[(i)] If two locally DC functions $H$ and $J$  are defined on the set $U,$ then $H+J$ and $H-J$ are also locally DC.
    \item[(ii)] If maps $F :U\to \R^k$ and $G : V\to \R^\ell$ are locally DC and $F(U)\subset V$,
    then $G\circ F : V\to \R^\ell$ is also locally DC.
    \item[(iii)] If the derivative of a function $f:U\to\R^k$ is locally Lipschitz on $U$ (or in particular if $f$ is smooth), then $f$ is locally DC.
\end{itemize}
\end{lemma}

\subsection{Lipschitz and delta-convex submanifolds}

We first give a precise definition of (embedded) Lipschitz submanifold of any codimension in the Euclidean case.
A map $f:X\to Y$ between metric spaces $(X,d_X),(Y,d_Y)$ is said to be \emph{bi-Lipschitz} if there is $L>0$ such that $L^{-1}d_X(x,y) \leq d_Y(f(x),f(y)) \leq Ld_X(x,y)$ holds for any $x,y\in X$.
Recall that any bi-Lipschitz map is a homeomorphism onto its image.

\begin{definition}[Lipschitz submanifold: Euclidean case]\label{def:Lipsubmfd_Euclid}
  Let $1\leq d\leq m-1$.
  A subset $S\subset\R^m$ is called \emph{$d$-dimensional Lipschitz submanifold} if for any $p\in S$ there are an open neighborhood $U$ of $p$ in $\R^m$ and a $d$-dimensional affine subspace $P$ of $\R^m$ such that if $\pi:S\to P$ denotes the orthogonal projection of $S$ to $P$, then the image $\pi(
  S\cap U)$ is open in $P$ and the restriction $\pi|_{S\cap U}$ defines a bi-Lipschitz map.
\end{definition}

Then we define DC submanifolds as a subclass of Lipschitz submanifolds.
A map $f:U\to\R^k$ on an open set $U\subset P^m$ in an $m$-dimensional affine subspace $P^m\subset\R^{m+\ell}$ is said to be DC if for some (in fact every) isometry $T$ on $\R^{m+\ell}$ with $T(\R^m\times\{\bo\})= P^m$, where $\bo$ denotes the origin, the map $f\circ T$ is DC
on the open subset $T^{-1}(U)$ of
$\R^m\times\{\bo\}\simeq\R^m$ (by interpreting $(x,\bo)\in \R^d\times\{\bo\}$ as $x\in \R^d$).





\begin{definition}[DC submanifold: Euclidean case]\label{def:DCsubmfd_Euclid}
  Let $1\leq d\leq m-1$.
  A subset $S\subset\R^m$ is called \emph{$d$-dimensional delta-convex (DC) submanifold} if for any $p\in S$ one can choose $U$ and $P$ in Definition \ref{def:Lipsubmfd_Euclid} with the additional property that the inverse map $(\pi|_{S\cap U})^{-1}:\pi(S\cap U)\to S\cap U$ is a DC map.
\end{definition}

In particular, a Lipschitz (resp.\ DC) submanifold of codimension one ($d=m-1$) is called \emph{Lipschitz} (resp.\ \emph{DC}) \emph{hypersurface}.
Hereafter we simply interpret a $0$-dimensional submanifold as a point.

\begin{remark}[Graphicality]
 Any Lipschitz submanifold $S\subset\R^m$ is by definition locally graphical in the sense that around each point of $S$, up to isometry, the set $S$ is locally of the form $$\left\{ (x',f_1(x'),\dots,f_{m-d}(x'))\in \R^m\mid x'=(x_1,\dots,x_d)\in U'\right\}$$
 for some Lipschitz functions $f_1,\dots,f_{m-d}$ on an open set $U'\subset\R^d$.
 If $S$ is DC, then it simply means that each function $f_j$ is DC.
\end{remark}

We will later give definitions of those submanifolds in a manifold that are independent of the choice of a local chart, see Section \ref{subsec:Lipsubmfd_Finsler}.

\subsection{Delta-convex hypersurface structure}

Recall that for $p\in\mathbf{R}^m\setminus N$,
$$d_N(p):=\inf_{q\in N}|q-p|.$$
The singular set $\Sigma(N)\subset\mathbf{R}^m\setminus N$ denotes the set of all nondifferentiable points of $d_N$.
We also recall the following key characterization:
Let $\pi_N:\mathbf{R}^m\setminus N \to 2^N$ be the projection map defined by
\begin{equation*}
  \pi_N(p):= \{q\in N \mid d_N(p)=|q-p|\}.
\end{equation*}
A well-known characterization (see e.g.\ \cite{Cannarsa2004}) is that
\begin{equation}\nonumber
  \Sigma(N)=\{p\in \R^m\setminus N \mid \#\pi_N(p)\geq2 \},
\end{equation}
where $\#$ denotes cardinality.
It is also well known and easy to prove that the map $\pi_N$ is set-valued upper semicontinuous, i.e., for any $p_j\to p$ in $M\setminus N$ and $\pi(p_j)\ni q_j\to q$ in $N$, we have $q\in\pi_N(p)$.
Finally, for an integer $k\geq2$ we introduce the subset $\Sigma_k(N)\subset\Sigma(N)$ of points that admit exactly $k$ nearest points to $N$:
\begin{equation}\label{eq:Sigma_2}
  \Sigma_k(N):=\{p\in\Sigma(N) \mid \#\pi_N(p)=k \}.
\end{equation}

The main goal of this section is to prove a more precise form of Theorem \ref{thm:main_singular_set} in the Euclidean case, given in terms of $\Sigma_2(N)$:

\begin{theorem}[DC hypersurface structure: Euclidean case]\label{thm:Euclidean_singular_set}
  Let $m\geq2$ and $N\subset\mathbf{R}^m$ be a closed subset.
  Then there exist at most countably many DC hypersurfaces $S_j\subset\R^m$ and $(m-2)$-dimensional DC submanifolds $S'_j\subset\R^m$ such that
  \begin{itemize}
    \item[(i)] $\Sigma_2(N) \subset S \subset \Sigma(N)$, where $S:=\bigcup_{j=1}^\infty S_j$,
    \item[(ii)] $\Sigma(N)\setminus S\subset\bigcup_{j=1}^\infty S'_j$
  \end{itemize}
\end{theorem}

A key tool for proving Theorem \ref{thm:Euclidean_singular_set} is Vesel\'{y}--Zaj\'{i}\v{c}ek's DC implicit function theorem.
To state it we first recall some notions in nonsmooth analysis for the class of locally Lipschitz functions (cf.\ \cite[
Chapter 2]{Clarke1990}), including all DC functions and hence distance functions.
For a function $f:\mathbf{R}^d\to\mathbf{R}$ locally Lipschitz around a point $x$, let $D^*f(x)$ denote the {\em reachable gradient} of $f$ at $x$, which is defined by the set of all vectors that are limits of (classical) gradients $Df$:
\begin{equation*}
  D^*f(x) := \Big\{ \lim_{j\to\infty}Df(x_j) \ \Big|\ x_j\to x,\ \exists Df(x_j) \Big\}.
\end{equation*}
Note that the set $D^*f(x)$ is nonempty by Rademacher's theorem, bounded since $|Df|\leq L$ holds for the (local) Lipschitz constant $L$, and also closed by its definition using limits.
In addition, let $\partial f(x)$ denote the {\em generalized gradient} defined by the convex hull of the reachable gradient:
\begin{equation}\label{eq:generalized_gradient}
  \partial f(x) := \conv D^*f(x),
\end{equation}
which is also compact, cf.\ \cite[Theorem 17.2]{Rockafellar1970}.
Recall that $f$ is Lipschitz near $x$ and $\partial f(x)$ is a singleton if and only if $f$ is strictly differentiable at $x$ \cite[Proposition 2.2.4]{Clarke1990}.
If $f$ is in addition semi-concave, then $\partial f(x)$ coincides with the supergradient \cite[Theorem 3.3.6]{Cannarsa2004}.
Since a (semi-)concave function is differentiable if and only if the supergradient is a singleton there, we find that for a (semi-)concave function the differentiability and the strict differentiability are equivalent.

Then we have the following key implicit function theorem:

\begin{theorem}[DC implicit function theorem {\cite[Proposition 5.9]{Vesely1989}}]\label{thm:DC_IFT}
  Let $V\subset\R^m$ be an open subset and $f_1,\dots,f_k:V\to\R$ be DC functions, where $1\leq k\leq m-1$, and let $x_0\in V$.
  If $f_1(x_0)=\dots=f_k(x_0)=0$ and 
  if $\partial f_1(x_0),\dots,\partial f_k(x_0)$ are singletons whose elements are linearly independent,
  then there exists a neighborhood $U\subset V$ of $x_0$ in $\R^m$ such that $S:=\bigcap_{j=1}^k f_j^{-1}(0)\cap U$ is an $(m-k)$-dimensional DC submanifold of $\mathbf{R}^m$.
\end{theorem}

\begin{remark}\label{rem:DC_IFT}
    The precise statement of \cite[Proposition 5.9]{Vesely1989} is the following: If $U\subset\R^n\times\R^k$ is an open set, $(a,b)\in U$ and $G:U\to\R^k$ is a locally DC mapping such that $G(a,b)=0$ and $\partial(G(a,\cdot))$ contains surjective mappings only, then there are neighborhoods $U_a\subset\R^n$ and $U_b\subset\R^k$ of $a$ and $b$, respectively, and a DC mapping $\varphi:U_a\to U_b$ such that for any $x\in U_a$ and $y\in U_b$, one has $G(x,y)=0$ if and only if $y=\varphi(x)$.
    Theorem \ref{thm:DC_IFT} can be regarded as a special case of \cite[Proposition 5.9]{Vesely1989} by taking $n=m-k$ and $G=(f_1,\dots,f_k)\circ A$ with a suitable $A\in O(m)$ so that $\partial(G(a,\cdot))$ has a single element of a regular $k\times k$ matrix.
\end{remark}

For Theorem \ref{thm:Euclidean_singular_set} we will only use the special case of $k=1$, but later also discuss the general case.

We also introduce a radius function to quantify singular points.
For each $p\in\Sigma(N)$ we define the radius function by
\begin{equation}\label{eq:rad}
  \rad_N(p):=\max_{q_1,q_2\in\pi_N(p)} |q_1-q_2|.
\end{equation}
In the Euclidean case, $\rad_N(p)>0$ holds for any $p\in\Sigma(N)$ (although this is not the case if $M$ is general).
Notice that if $p\in\Sigma_2(N)$, then $\rad_N(p)$ denotes exactly the distance of two points in $\pi_N(p)$.

Now we prove the key fact that from any point in $\Sigma_2(N)$ a hypersurface propagates along $\Sigma(N)$, with the quantitative cleaving property that the residual part has relatively small radii.

\begin{lemma}\label{lem:biLipschitzpropagate}
  For any $p\in \Sigma_2(N)$, there exist a positive number $\delta(p)>0$ and a DC hypersurface $S_p\subset\mathbf{R}^m$ such that $S_p\subset\Sigma(N)$ and such that $\rad_N(y)\leq \frac{1}{2}\rad_N(p)$ holds for any $y\in (\Sigma(N)\cap B_{\delta(p)}(p))\setminus S_p$, where $B_r(p):=\{x\in \mathbf{R}^m \mid |p-x|<r\}$.
\end{lemma}

\begin{proof}
  Let $\pi_N(p)=\{q_1,q_2\}$ and $N_j:=N\cap\{ y\in \mathbf{R}^m \mid |q_j-y|\leq \frac{1}{4}\rad_N(p) \}$, where $j=1,2$.
  The sets $N_1$ and $N_2$ are compact.
  In addition, by $\rad_N(p)=|q_1-q_2|>0$, we have
  \begin{equation}\label{eq10}
    N_1\cap N_2=\emptyset,
  \end{equation}
  and also $N_1$ and $N_2$ are (relative) neighborhoods of $q_1$ and $q_2$ in $N$, respectively.

  Let $f:=d_{N_1}-d_{N_2}=d(N_1,\cdot)-d(N_2,\cdot)$.
  Since $\pi_N$ is set-valued upper semicontinuous, there is $r>0$ such that for any $x\in B_r(p)$ we have $\pi_N(x)\subset N_1\cup N_2$, and in particular $d_N(x)=\min\{d_{N_1}(x),d_{N_2}(x)\}$.
  Hence, if $x\in B_r(p)$ and $f(x)=0$, then $d_{N_1}(x)=d_{N_2}(x)=d_N(x)$ and hence, by \eqref{eq10}, $\#\pi_N(x)\geq2$.
  Therefore, for any $\delta\in(0,r]$,
  \begin{equation*}
    S_{p,\delta} := f^{-1}(0)\cap B_\delta(p)\subset \Sigma(N).
  \end{equation*}
  In addition, if $x\in(\Sigma(N)\cap B_\delta(p))\setminus S_{p,\delta}$, then either $\pi_N(x)\subset N_1$ or $\pi_N(x)\subset N_2$ holds (depending on the sign of $f(x)$) and hence $\rad_N(x)\leq\frac{1}{2}\rad_N(p)$ holds by definition of $N_j$ and the triangle inequality.

  We finally prove that for a suitable positive number $\delta:=\delta(p)\in(0,r]$ the set $S_p:=S_{p,\delta(p)}$ is a DC hypersurface.
  The function $f=d_{N_1}-d_{N_2}$ is DC by Lemma \ref{lem:DC_fundamental_properties} (i).
  Since $p$ admits a unique $N_j$-segment from $q_j$ for $j=1,2$, letting $v_j:=\frac{p-q_j}{|p-q_j|}$, we have $\partial d_{N_j}(p)=\{v_j\}$ (cf.\ \cite[Corollary 3.4.5 (iv)]{Cannarsa2004}) and in particular $d_{N_j}$ is differentiable at $p$.
  Hence $\partial f(p)=\{v_1-v_2\}$; indeed, $f$ is also differentiable at $p$ and $D f(p)=v_1-v_2$ so that $v_1-v_2\in\partial f(p)$, while from the general inclusion property we obtain $\partial f(p)\subset \partial d_{N_1}(p) - \partial d_{N_2}(p)=\{v_1-v_2\}$.
  Since $0\not\in\partial f(p)$ by $v_1\neq v_2$, and since $f(p)=0$, Theorem \ref{thm:DC_IFT} with $k=1$ implies that there is a small positive number $\delta(p)>0$ such that $S_p=f^{-1}(0)\cap B_{\delta(p)}(p)$ is a DC hypersurface.
\end{proof}

Lemma \ref{lem:biLipschitzpropagate} claims that from any point $p$ of $\Sigma_2(N)$ a DC hypersurface $S_p$ propagates along $\Sigma(N)$.
It looks simple but is somewhat delicate: It is claimed neither that $S_p\subset\Sigma_2(N)$ nor that $\Sigma_2(N)\cap B_{\delta(p)}(p)\subset S_p$.
The following example is helpful for understanding these facts, cf.\ Figure \ref{fig:branch}:

\begin{example}[Branched singular set]\label{ex:branch}
  Define a closed set $N\subset\mathbf{R}^2$ by
  $$N:=\{(\pm1,\tfrac{1}{k})\in\mathbf{R}^2 \mid k\in\mathbf{Z}_{>0} \} \cup \{(\pm1,0)\}.$$
  Then, letting $a_k:=\frac{1}{2}(\frac{1}{k}+\frac{1}{k+1})$ and
  $$X:=\{ (x,a_k)\in\mathbf{R}^2 \mid x\in\mathbf{R},\ k\in\mathbf{Z}_{>0} \},\quad Y:=\{(0,y)\in\mathbf{R}^2 \mid y\in\mathbf{R}\},$$
  we can explicitly represent the singular sets $\Sigma(N)$ and $\Sigma_2(N)$:
  $$\Sigma(N) = X \cup Y, \quad \Sigma(N)\setminus\Sigma_2(N) = X \cap Y =\{(0,a_k)\in\mathbf{R}^2 \mid k\in\mathbf{Z}_{>0} \}.$$
  Note that $\rad_N(p)\leq\frac{1}{2}$ for $X\setminus Y$, while $\rad_N(p)\geq2$ for $p\in Y$.
  Now we look at the origin $p=(0,0)$.
  Notice that $p\in\Sigma_2(N)$, and hence a Lipschitz graph $S_p$ propagates along $\Sigma(N)$ by Lemma \ref{lem:biLipschitzpropagate}.
  In addition, since $\rad_N(p)=2$, by the radius estimate in Lemma \ref{lem:biLipschitzpropagate} we deduce that $Y\cap B_{\delta(p)}(p) \subset S_p \cap B_{\delta(p)}(p)$; by the bijectivity for $S_p$ we thus find that $Y\cap B_{\delta(p)}(p) = S_p\cap B_{\delta(p)}(p)$, i.e., the curve $S_p$ must be a vertical segment along $Y$ in $B_{\delta(p)}(p)$.
  Therefore, this curve $S_p$ passes through infinitely many points of $\Sigma(N)\setminus\Sigma_2(N)=X\cap Y$ so that $S_p\not\subset\Sigma_2(N)$, and also there remain infinitely many branches $(\Sigma_2(N)\cap B_{\delta(p)}(p))\setminus S_p = (X\setminus Y)\cap B_{\delta(p)}(p)$ so that $\Sigma_2(N)\cap B_{\delta(p)}(p)\not\subset S_p$.
\end{example}

\begin{center}
    \begin{figure}[htbp]
      \includegraphics[width=200pt]{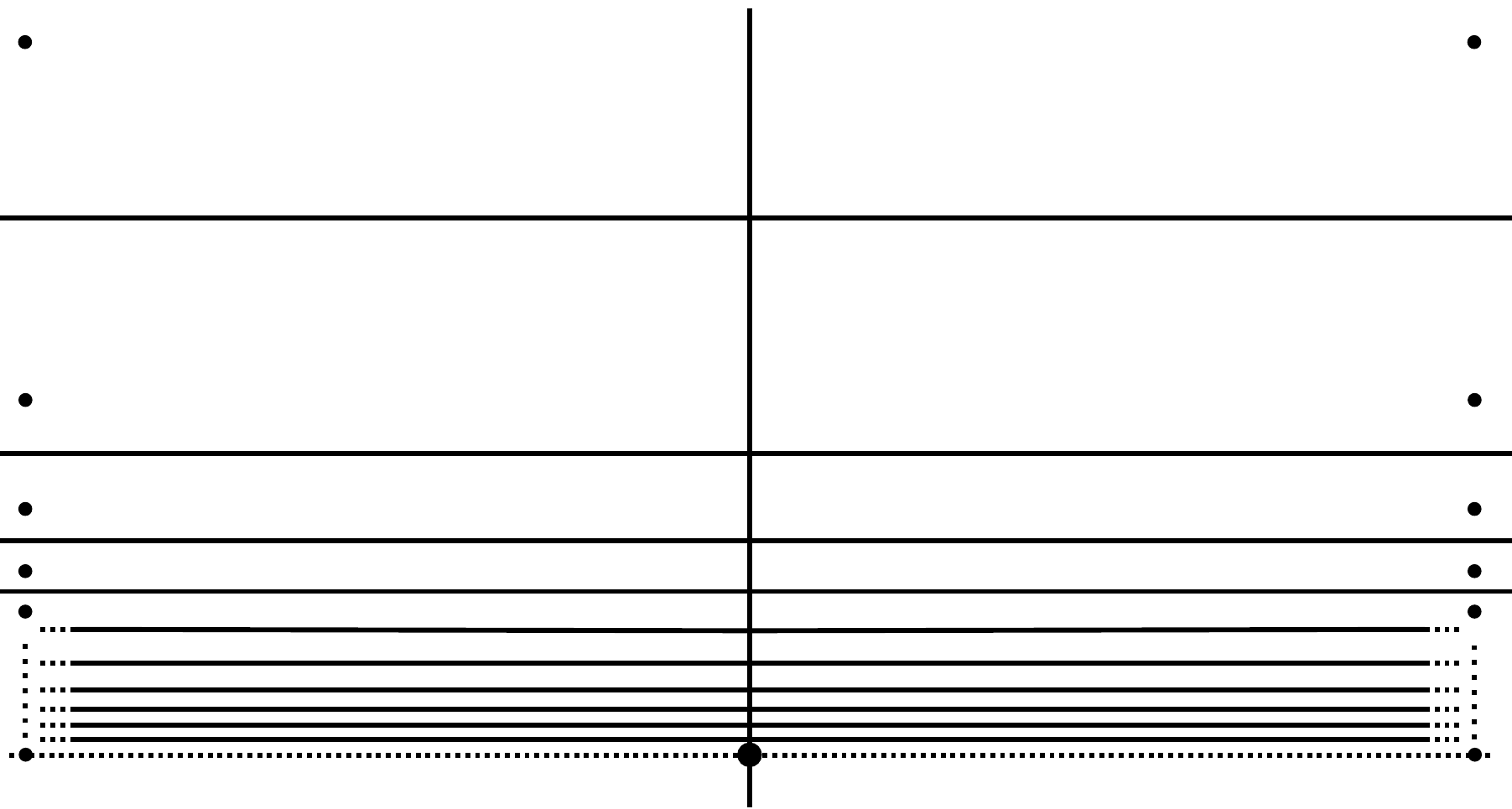}
      \caption{Branched singular set.}
      \label{fig:branch}
  \end{figure}
\end{center}

With Lemma \ref{lem:biLipschitzpropagate} at hand, we are now in a position to prove Theorem \ref{thm:Euclidean_singular_set}.

\begin{proof}[Proof of Theorem \ref{thm:Euclidean_singular_set}]
  We first prove assertion (i), that is, there is an at most countable union $S$ of DC hypersurfaces such that $\Sigma_2(N)\subset S\subset\Sigma(N)$, and then (independently) verify that $\Sigma(N)\setminus\Sigma_2(N)$ is covered by a countable union of DC submanifolds of codimension two.
  These two claims immediately imply assertion (ii) since $\Sigma(N)\setminus S \subset \Sigma(N)\setminus\Sigma_2(N)$.

  {\em Step 1: Countable covering by DC hypersurfaces.}
  We may suppose that $\Sigma_2(N)\neq\emptyset$ without loss of generality since otherwise Step 2 directly completes the proof.
  Let $K\subset\mathbf{R}^m$ be any compact subset such that $\Sigma_2(N)\cap K\neq\emptyset$.
  We define
  \begin{equation}\label{eq:rad_K}
    \rad_N(K):=\max\{ \rad_N(p) \mid p\in \Sigma(N)\cap K\}.
  \end{equation}
  For each positive integer $i$,
  \begin{equation}\label{eq:A_i}
    A_i := \left\{ p\in\Sigma_2(N)\cap K \ \Big|\ \tfrac{1}{i+1}\rad_N(K) < \rad_N(p) \leq \tfrac{1}{i}\rad_N(K) \right\}.
  \end{equation}
  Finally, for each positive integers $i,j$, we define
  $$A'_{ij}:=\{p\in A_i \mid \delta(p)\geq 1/j \},$$
  where $\delta(p)$ is a possible choice in Lemma \ref{lem:biLipschitzpropagate}.

  We prove that for each $i,j$ the set $A'_{ij}$ is covered by an at most finite union of DC hypersurfaces contained in $\Sigma(N)$.
  We may assume that $A'_{ij}$ is nonempty since otherwise obvious.
  Take an arbitrary point $p_1\in A'_{ij}$, and also a number $\delta(p_1)>0$ and a DC hypersurface $S_{p_1}\subset\Sigma(N)$ in Lemma \ref{lem:biLipschitzpropagate}.
  Then, by definition of $A_i$ and by the radius estimate in Lemma \ref{lem:biLipschitzpropagate}, $(\Sigma(N)\cap B_{\delta(p_1)}(p_1))\setminus S_{p_1}$ and $A_i$ are disjoint.
  Therefore,
  $$A'_{ij}\setminus S_{p_1} \subset A_i\setminus S_{p_1} \subset K\setminus B_{\delta(p_1)}(p_1).$$
  If $A'_{ij}\setminus S_{p_1}$ is empty, then the proof is complete, while if nonempty, then we can choose a point $p_2\in A'_{ij}\setminus S_{p_1}$ and a parallel procedure implies that
  $$\textstyle A'_{ij}\setminus\bigcup_{k=1}^2 S_{p_k}\subset K\setminus\bigcup_{k=1}^2 B_{\delta(p_k)}(p_k).$$
  If the left-hand side is nonempty, then we choose $p_3$ and continue this procedure.
  This inductive procedure terminates in finite steps; indeed, at step $\ell$ we need to choose a point $p_{\ell}$ included in the set $K\setminus\bigcup_{k=1}^{\ell-1}B_{\delta(p_k)}(p_k)$, which becomes empty in finite steps since $K$ is compact, and since $\delta(p_k)\geq1/j$ holds for all $k$ by definition of $A'_{ij}$ and also $p_k\in K\setminus\bigcup_{k'=1}^{k-1}B_{\delta(p_{k'})}(p_{k'})\subset K\setminus\bigcup_{k'=1}^{k-1}B_{1/j}(p_{k'})$.
  Therefore, $A'_{ij}$ has the desired finite-covering.

  Then, since $A_i=\bigcup_{j=1}^\infty A'_{ij}$, the set $A_i$ is covered by an at most countable union of DC hypersurfaces contained in $\Sigma(N)$.
  In addition, since $\rad_N(p)>0$ holds for any $p\in\Sigma(N)$, it is clear that $\Sigma_2(N)\cap K=\bigcup_{i=1}^{\infty} A_i$.
  Therefore, $\Sigma_2(N)\cap K$ also has the same type of countable-covering.
  Finally, since $K$ is arbitrary, by taking a countable sequence of closed balls $B_k$ of radius $k$ we deduce that $\Sigma_2(N)=\bigcup_{k=1}^\infty\Sigma_2(N)\cap B_k$ also has the desired countable-covering $S$.

  {\em Step 2: Codimension $2$ covering of the residual part.}
  For any $p\in\Sigma(N)\setminus\Sigma_2(N)$ there are at least three distinct points $q_1,q_2,q_3\in\pi_N(p)$ contained in a round sphere of radius $d_N(p)$ centered at $p$.
  Hence the set $\Sigma(N)\setminus\Sigma_2(N)$ is contained in the set $\{p\in \R^m\setminus N \mid \dim\conv(\pi_N(p))\geq2 \}$ (which is equivalent to $\Sigma^2(d_N)$).
  Therefore, by \cite[Theorem 8]{Zajicek1983}, the set $\Sigma(N)\setminus\Sigma_2(N)$ is covered by a countable union of $(m-2)$-dimensional DC submanifolds.
\end{proof}

\subsection{Optimality of DC regularity}

Theorem \ref{thm:Zajicek-Pavlica} shows that the delta-convexity is indeed optimal regularity for the singular set of a general semi-concave function.
However it is not yet clear if the DC regularity in Theorem \ref{thm:main_singular_set} (or Theorem \ref{thm:2D_singular_set}) is also optimal, since the class of distance functions is strictly smaller than the class of semi-concave functions.
Here we construct an example which ensures that our DC assertion cannot be replaced with semi-concavity nor $C^1$-regularity.
A key point is that the class of locally DC functions on an interval is equivalent to the class of functions whose derivatives are locally of bounded variation (cf.\ \cite{LionsIUMJ}), while any semi-concave (resp.\ semi-convex) function can have derivative with only downward (resp.\ upward) jumps.

Hereafter we frequently use the following notation: For $g:[a,b]\to\R$,
\begin{equation*}
    \textup{Graph}(g):=\{(x,g(x))\in\R^2 \mid x\in[a,b] \}.
\end{equation*}

\begin{proposition}\label{prop:optimality}
    For any $\varepsilon\in(0,1)$ there exists a bounded open set $\Omega\subset\R^2$ with $C^{1,1}$-boundary such that the corresponding singular set $\Sigma(\R^2\setminus\Omega)$ is represented by
    $$\Sigma(\R^2\setminus\Omega) = \textup{Graph}(f),$$
    where $f:[0,1]\to\R$ is a DC function with the following property: There is a closed subset $C\subset [0,1]$ with measure $|C|\geq 1-\varepsilon$ such that for any $x\in C$ and any $\delta>0$, the restricted function $f|_{B_\delta(x)\cap [0,1]}$, where $B_\delta(x):=(x-\delta,x+\delta)$, has derivative with infinitely many upward and downward jumps.
    In particular, $f|_{B_\delta(x)\cap [0,1]}$ is not $C^1$ nor locally semi-concave nor semi-convex.
\end{proposition}

\begin{proof}
    Let $\Omega_0\subset\R^2$ be a bounded $C^{1,1}$ domain defined by
    $$\Omega_0:=\{p\in\R^2 \mid d(S,p)<1 \},\quad S=[0,1]\times\{0\}\subset\R^2.$$
    Note that $\Sigma(\R^2\setminus\Omega_0)=S$, and the boundary $\partial\Omega_0$ can be decomposed as
    \begin{align*}
        \partial\Omega_0 &= F_+\cup F_-\cup R_1\cup R_2,\\
        F_+ &= [0,1]\times\{1\},\ F_- = [0,1]\times\{-1\},\\
        R_1 &= \{(x,y)\in \R^2 \mid x^2+y^2=1,\ x<0 \},\\
        R_2 &= \{(x,y)\in \R^2 \mid (x-1)^2+y^2=1,\ x>1 \}.
    \end{align*}
    We construct a domain $\Omega$ with the desired properties by modifying the flat part $F_\pm$ of the boundary of $\Omega_0$, while keeping the round part $R_1,R_2$.

    \textit{Step 1: Base function.}
    Given small $\delta>0$, say $\delta\in(0,\frac{1}{100})$, we consider the following even $C^{1,1}$ function composed of the line $\{y=1\}$ and a circular arc of radius $1+\delta$ centered at $(0,\delta)$ and two circular arcs of radius $1$ centered at $(\pm a_\delta,2)$, respectively:
    \begin{equation*}
        g_\delta(x):=
        \begin{cases}
            \delta+\sqrt{(1+\delta)^2-x^2} & (|x|\leq b_\delta),\\
            2-\sqrt{1-(x-a_\delta)^2} & (b_\delta \leq x \leq a_\delta),\\
            2-\sqrt{1-(x+a_\delta)^2} & (-a_\delta \leq x \leq -b_\delta),\\
            1 & (|x|\geq a_\delta),
        \end{cases}
    \end{equation*}
    where $a_\delta>0$ denotes a unique positive number such that the two elementary functions $\delta+\sqrt{(1+\delta)^2-x^2}$ and $2-\sqrt{1-(x-a_\delta)^2}$ have a (unique) point of tangency at $x=b_\delta$ with $b_\delta\in(0,a_\delta)$ (see also Figure \ref{fig:bump}).
    It is easy to see that
    \begin{equation}\label{eq:1030-1}
        a_\delta=2\sqrt{2\delta},\quad \lim_{\delta\to0}\|g_\delta-1\|_{C^1}=0,\quad 
        \sup_{\delta\in(0,\frac{1}{100})}\|g_\delta''\|_{L^\infty}<\infty.
    \end{equation}

    \begin{center}
        \begin{figure}[htbp]
          \includegraphics[width=200pt]{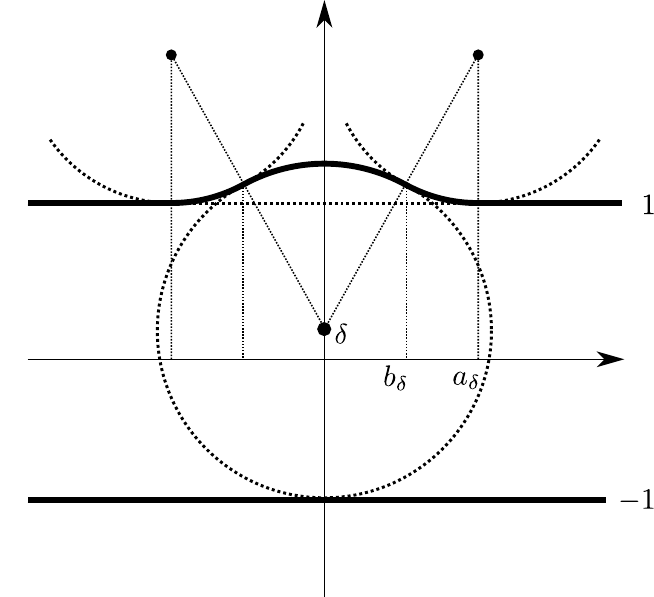}
          \includegraphics[width=180pt]{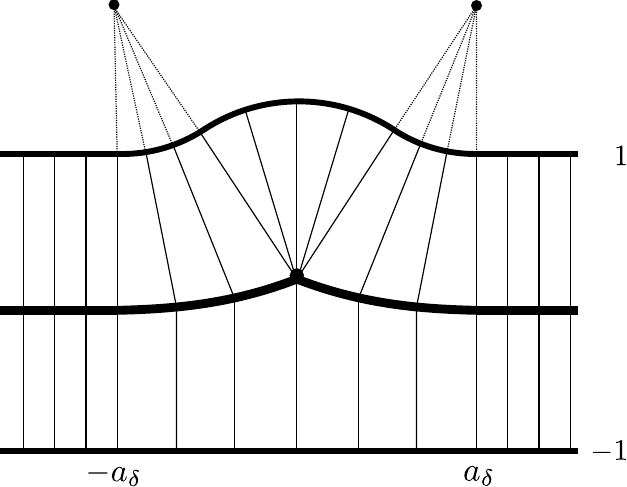}
          \caption{Bumped boundary and the corresponding singular set.}
          \label{fig:bump}
      \end{figure}
    \end{center}

    \textit{Step 2: Base construction.}
    To any interior point $x_0\in(0,1)$ and any $r_0>0$ such that $B_{r_0}(x_0)\subset (0,1)$, we associate $h_0:[0,1]\to\R$ by 
    $$h_0(x):=g_{\delta_0}(x-x_0),$$
    where $\delta_0>0$ is chosen to be so small that
    \begin{equation}\label{eq:jump2}
        a_{\delta_0}<r_0.
    \end{equation}
    Using this function, if we replace the part $F_+$ of $\partial\Omega_0$ with $\textup{Graph}(h_0)$, then the singular set is partly replaced by two pieces of parabola (as in Figure \ref{fig:bump}).
    More precisely, if we let $\Omega_0'$ denote the unique bounded domain such that
    $$\partial\Omega_0'=\textup{Graph}(h_0)\cup F_-\cup R_1\cup R_2,$$
    then $\partial\Omega_0'$ is of class $C^{1,1}$, and also we deduce that
    $$\Sigma(\R^2\setminus\Omega_0')=\textup{Graph}(f_0),$$
    where $f_0:[0,1]\to\R$ is defined by
    \begin{equation*}
        f_0(x):=
        \begin{cases}
            \frac{1}{8}(x-x_0-a_{\delta_0})^2 & (0\leq x-x_0 \leq a_{\delta_0}),\\
            \frac{1}{8}(x-x_0+a_{\delta_0})^2 & (-a_{\delta_0}\leq x-x_0 \leq 0),\\
            0 & (|x-x_0|\geq a_{\delta_0}),
        \end{cases}
    \end{equation*}
    by the following elementary geometric argument:
    The locus of the center of circles touching both the upper and lower boundary is clearly graphical.
    The locus is obviously horizontal whenever $|x-x_0|\geq a_\delta$.
    On the other hand, if $0\leq x-x_0\leq a_\delta$ or if $-a_\delta \leq x-x_0\leq 0$, the locus must be a suitable piece of parabola since it is a set of points equidistant from the line $\{y=-1\}$ and one of the circular arcs of radius $1$ centered at $(\pm a_\delta,2)$, or equivalently, equidistant from the line $\{y=-2\}$ and one of the points $(\pm a_\delta,2)$.
    In addition, geometrically, those parabolas must be tangent to the $x$-axis at $x=x_0\pm a_0$.
    Finally, the prefactor $\frac{1}{8}$ can be computed by using $f_0(x_0)=\delta$ and $a_\delta=2\sqrt{2\delta}$.
    
    Note that the derivative $f_0'$ has a downward jump at $x=x_0$, namely,
    \begin{equation}\label{eq:jump}
        f_0'(x_0+0)-f_0'(x_0-0)=-\frac{a_{\delta_0}}{2}.
    \end{equation}

    \textit{Step 3: Next construction.}
    Given an open interval $J\subset[0,1]$, we replace the part $F_\pm$ with certain new sets $G_\pm^J$ such that $F_\pm\setminus(J\times\R)=G_\pm^J\setminus(J\times\R)$ in the following way.
    
    Take a sequence $\{x_j\}_j\subset J$ and $\{r_j\}_j\in(0,1)$ such that  $B_{r_j}(x_j)$, $j=1,2,\dots$, are mutually disjoint subsets in $J$, and the subsequences $\{x_{3j}\}_j$ and $\{x_{3j+1}\}_j$ converge to the left endpoint and the right endpoint of $J$, respectively.
    Now, to each $j$, we associate functions $h_j$ and $f_j$ as in Step 2; namely $h_j:=h_0$ and $f_j:=f_0$ in Step 2 with $x_0:=x_j$ and $r_0:=r_j$.
    We also choose $\delta_j\in(0,\frac{1}{100})$ such that $a_{\delta_j}<r_j$ in view of \eqref{eq:jump2}.
    Note that $r_j\to0$ and $\delta_j\to0$.
    Then we define two functions
    $h^\pm_J:[0,1]\to\R$ by
    \begin{equation*}
        h_+^J(x):= \max_j h_{2j}(x), \quad h_-^J(x):= -\max_j h_{2j+1}(x). 
    \end{equation*}
    Clearly, $\pm h_\pm^J\geq 1$.
    In addition, we deduce that $\|h_\pm^J\|_{C^{1,1}}<\infty$ as follows: We only argue for $h_+^J$.
    Let $u_j:=h_{2j}-1$.
    Since $u_j\geq0$ and the supports of $u_j$'s are mutually disjoint, we can represent $h_+^J=1+\sum_{j=1}^\infty u_j$, and also we have $\|\sum_{j=N}^{N+M}u_j\|_{C^1}=\max_{0\leq k\leq M}\|u_{N+k}\|_{C^1}$.
    Then by \eqref{eq:1030-1} we obtain 
    $$\lim_{N\to\infty}\sup_{k\geq0}\|u_{N+k}\|_{C^1}=\lim_{N\to\infty}\sup_{k\geq0}\|g_{\delta_{N+k}}-1\|_{C^1}=0.$$
    Thus the partial sum $U_N:=\sum_{j=1}^Nu_j$ is convergent in $C^1$ as $N\to\infty$, and hence $h_+^J\in C^1$.
    In addition, again using the disjoint-support property of $u_j$'s, we have $\|U_N''\|_{L^\infty}=\max_{1\leq j\leq N}\|u_j''\|_{L^\infty}$.
    Hence for any $x,y\in [0,1]$,
    \begin{align*}
        |U'_N(x)-U'_N(y)|\leq \|U_N''\|_{L^\infty}|x-y|&=\max_{1\leq j\leq N}\|u_j''\|_{L^\infty}|x-y|\\
        &\leq \sup_{\delta\in(0,\frac{1}{100})}\|g_\delta''\|_{L^\infty}|x-y|.
    \end{align*}
    Letting $N\to\infty$, this with \eqref{eq:1030-1} implies that $|(h_+^J)'(x)-(h_+^J)'(y)|\leq C|x-y|$ for some universal $C>0$, and therefore $\|h_+^J\|_{C^{1,1}}<\infty$.
    
    Now we replace $F_\pm$ with
    $$G_\pm^J:=\textup{Graph}(h_\pm^J).$$
    Since the deformation of the boundary at step $j$ only affects the singular set in the region $B_{r_j}(x_j)\times\R$ (mutually disjoint for all $j$), if we let $\Omega^J$ denote the unique bounded domain such that
    $$\partial\Omega^J=G_+^J\cup G_-^J\cup R_1\cup R_2,$$
    then $\partial\Omega^J$ is of class $C^{1,1}$ and the singular set is of the form
    $$\Sigma(\R^2\setminus\Omega^J)=\textup{Graph}( f^J ), \quad f^J:=\max_j f_{2j}-\max_j f_{2j+1}:[0,1]\to\R,$$
    and the derivative of $f^J$ has infinitely many jumps; more precisely, it has jumps exactly at all $x_j$, downward at $x_{2j}$ and upward at $x_{2j+1}$; in particular, the subsequence $\{x_{6j}\}_j$ (resp.\ $\{x_{6j+3}\}_j$) consists of downward (resp.\ upward) jump points converging to the left endpoint of $J$, and also $\{x_{6j+4}\}_j$ (resp.\ $\{x_{6j+1}\}_j$) consists of downward (resp.\ upward) jump points converging to the right endpoint of $J$.

    \textit{Step 4: Final construction.}
    For any given $\varepsilon>0$ we take the (standard) fat Cantor set $C\subset[0,1]$ with measure $|C|\geq1-\varepsilon$ constructed by deleting mutually disjoint open subintervals $\{J_i\}_i$ from $[0,1]$.
    Let $h_\pm^{J_i},f^{J_i}:[0,1]\to\R$ are the functions constructed in Step 3 (with $J=J_i$).
    We define
    $$G_\pm:=\textup{Graph}(h_\pm), \quad h_+:=\max_i h_+^{J_i}, \quad h_-:=\min_i h_-^{J_i},$$
    and consider the unique bounded domain $\Omega$ such that 
    $$\partial\Omega=G_+\cup G_-\cup R_1\cup R_2.$$
    Similarly to Step 3, the deformation at each step $i$ only affects in the region $J_i\times\R$ (mutually disjoint for all $i$), and hence the boundary is still of class $C^{1,1}$ and we have
    $$\Sigma(\R^2\setminus\Omega)=\textup{Graph}( f ), \quad f:=\sum_i f^{J_i}:[0,1]\to\R.$$
    
    We finally show that the function $f$ has the desired properties.
    The fact that $f$ is DC follows by a direct computation that $f'$ is of bounded variation; indeed, the sum of all the jumps of $f'$ is finite thanks to \eqref{eq:jump2}, \eqref{eq:jump}, and the fact that $\sum_{j}r_j\leq |J|$ in Step 3 and $\sum_{i}|J_i|\leq \varepsilon$ in Step 4, while except at the jump points $|f''|\leq\frac{1}{4}$.
    Fix any $x\in C$ and $\delta>0$.
    Since $C$ is closed and nowhere dense (so has empty interior), the point $x$ can be approached by points in $[0,1]\setminus C$.
    Therefore the set $B_\delta(x)\cap[0,1]$ contains an endpoint $\bar{x}$ of a connected component $J_i$ of the set $[0,1]\setminus C$, and hence the restricted function $f|_{B_\delta(x)\cap[0,1]}$ has derivative with infinitely many jumps around $\bar{x}$, both upward and downward.
    The proof is complete.
\end{proof}

\subsection{Higher codimensional DC graphical propagation}

In a very parallel way to Lemma \ref{lem:biLipschitzpropagate}, we can also deduce a (conditional) DC graphical propagation of higher codimension.

\begin{theorem}[Higher codimensional propagation]\label{thm:high_biLipschitz}
  Let $p\in\Sigma_{k+1}(N)$ with $1\leq k\leq m$.
  Suppose that $p\in\Sigma_{k+1}(N)$ and $\dim\conv(\pi_N(p))=k$.
  Then there is an $(m-k)$-dimensional DC submanifold $S$ of $\R^m$ such that $p\in S\subset \Sigma(N)\setminus\bigcup_{j=2}^k\Sigma_j(N)$.
\end{theorem}


\begin{proof}
  Since $p\in\Sigma_{k+1}(N)$, let $\pi_N(p)=\{q_0,\dots,q_k\}$.
  Let $N_j:=N\cap\{ y\in \mathbf{R}^m \mid |q_j-y|\leq \frac{1}{4}\min_{i\neq j}|q_i-q_j| \}$.
  The sets $N_0,\dots,N_k$ are compact and, by definition,
  \begin{equation}\label{eq:20}
    N_i\cap N_j=\emptyset \quad \text{for all $0\leq i<j\leq k$}.
  \end{equation}
  Let $f_j:=d_{N_0}-d_{N_j}=d(N_0,\cdot)-d(N_j,\cdot)$ for $1\leq j\leq k$.
  Since $\pi_N$ is set-valued upper semicontinuous, there is $r>0$ such that for any $x\in B_r(p)$ we have $\pi_N(x)\subset \bigcup_{j=0}^k N_j$.
  In particular, if $x\in B_r(p)$, then $d_N(x)=\min_{0\leq j\leq k}d_{N_j}(x)$, and hence if in addition $f_1(x)=\dots=f_k(x)=0$, then by \eqref{eq:20}
  we have $d_N(x)=d_{N_0}(x)=\dots=d_{N_k}(x)$ and $\#\pi_N(x)\geq k+1$, that is, $x\in\Sigma(N)\setminus\bigcup_{j=2}^k\Sigma_j(N)$.
  Hence, for any $\delta\in(0,r]$,
  \begin{equation*}
    S_{p,\delta} := \bigcap_{j=1}^k f_j^{-1}(0)\cap B_{\delta}(p) \subset \Sigma(N)\setminus\bigcup_{j=2}^k\Sigma_j(N).
  \end{equation*}
  Since $\pi_{N_j}(p)=\{q_j\}$ for each $j=0,\dots,k$, letting $v_j:=\frac{p-q_j}{|p-q_j|}$, we have $\partial d_{N_j}(p)=\{v_j\}$ and hence $\partial f_j(p)=\{v_0-v_j\}$.
  By the assumption that $\dim\conv(\pi_N(p))=k$, the unique element in $\partial f_1(p)\times\dots\times\partial f_k(p)$ is linearly independent.
  Combining this fact with $f_1(p)=\dots=f_k(p)=0$, we can apply Theorem \ref{thm:DC_IFT} to deduce that there is $\delta(p)>0$ such that $S_{p,\delta(p)}$ is an $(m-k)$-dimensional DC submanifold.
\end{proof}

The assumption in Theorem \ref{thm:high_biLipschitz} is equivalent to that the elements of $\pi_N(p)$ form a $k$-simplex.

Theorem \ref{thm:high_biLipschitz} strengthens the assertion of \cite[Corollary 6.4]{Albano1999} by imposing the stronger $k$-simplicity assumption.
The assumption that $\dim\conv(\pi_N(p))=k$ may be removed for $k=1,2$ as it is automatically satisfied; indeed, the case $k=1$ is nothing but Lemma \ref{lem:biLipschitzpropagate}, while for $k=2$, the three elements of $\pi_N(p)$ lie in a round sphere and hence cannot be colinear.
In particular, from any point in $\Sigma_3(N)$ a codimension-two DC submanifold propagates within $\Sigma(N)\setminus\Sigma_2(N)$.

On the other hand, Theorem \ref{thm:high_biLipschitz} is not sufficient for covering all parts of codimension two or higher.
For example, there is a simple example of $N$ such that any point in $\Sigma(N)\setminus\Sigma_2(N)$ does not satisfy the assumption of Theorem \ref{thm:high_biLipschitz}.

\begin{example}[Generic non-simplicial singularities]
  Let $m=3$.
  Let $k\geq3$, and $P\subset\R^2$ be the domain bounded by a regular $(k+1)$-gon centered at the origin.
  Consider the closed set $N:=(\R^2\setminus P)\times\R\subset\R^3$.
  Then by a simple observation we infer that $\Sigma(N)\setminus\Sigma_2(N)=\Sigma_{k+1}(N)=\{(0,0)\}\times\R\subset\R^3$, being a codimension-two submanifold.
  In this case all points $p\in\Sigma(N)\setminus\Sigma_2(N)=\Sigma_{k+1}(N)$ are not $k$-simplicial since $\pi_N(p)$ is contained in a plane.
  (If we replace $P$ with a disk, then we also notice that the set $\Sigma_\infty(N):=\{\#\pi_N(p)=\infty\}$ may not be negligible.)
\end{example}


In particular, the following problem (weaker than Problem \ref{prob:DCcovering}) is remained open:

\begin{problem}\label{prob:highercodimension}
  Let $m\geq3$ and $N\subset\R^m$ be a nonempty closed subset.
  Let $k$ be an integer such that $2\leq k\leq m-1$.
  Are there a countable family $\{S_j\}_{j=1}^\infty$ of DC submanifolds $S_j\subset\R^m$ of dimension $\geq m-k$, and a set $R\subset\Sigma(N)$ covered by a countable union of $(m-k-1)$-dimensional DC submanifolds, such that $\Sigma(N)=\bigcup_{j=1}^\infty S_j\cup R$?
\end{problem}

\subsection{Comparison with semi-concave function theory}\label{subsec:semiconcave}

As a concluding remark in the Euclidean case we compare our Theorem \ref{thm:Euclidean_singular_set} for distance functions with general theory of semi-concave functions.

In particular, we precisely recall a closely related result of Albano--Cannarsa \cite{Albano1999}.
In this part we let $\bd[A]$ denote the topological boundary of a subset $A\subset\mathbf{R}^m$ (to avoid confusion with the generalized gradient).
Let $K_C(x)$ denote the normal cone of a convex set $C\subset\mathbf{R}^m$ at $x\in C$ defined by $K_C(x):=\{q\in\mathbf{R}^m\mid \langle q,y-x \rangle\leq0,\ \forall y\in C\}$.
Albano--Cannarsa's result \cite[Theorem 5.1]{Albano1999} states that if a locally semi-concave function $u:\Omega\to\mathbf{R}$ has a point $x_0\in\Omega$ such that
\begin{equation}\label{eq:AlbanoCannarsa}
  \bd[\partial u(x_0)]\setminus D^*u(x_0)\neq\emptyset,
\end{equation}
then for any $p_0\in\bd[\partial u(x_0)]\setminus D^*u(x_0)$ there exist an open ball $B_\rho$ centered at the origin and a Lipschitz function $F:K_{\partial u(x_0)}(p_0)\cap B_\rho \to\Sigma(u)$ of the form
\begin{equation*}
  F(q)=x_0-q+o(|q|) \quad (|q|\to0).
\end{equation*}
This ``Lipschitz propagation'' already implies fine lower bounds such as
$$\inf\big\{\mathrm{diam}\big(\partial u(F(q))\big) \mid q\in K_{\partial u(x_0)}(p_0)\cap B_\rho \big\}>0,$$
and, for $\nu:=\dim K_{\partial u(x_0)}(p_0)$,
$$\liminf_{r\to+0}r^{-\nu}\mathcal{H}^\nu\big( F(K_{\partial u(x_0)}(p_0)\cap B_\rho)\cap B_r(x_0) \big)>0 \quad \Big(\Rightarrow \dim_{\mathcal{H}}\Sigma(u)\geq\nu \Big).$$
Condition \eqref{eq:AlbanoCannarsa} looks somewhat delicate, but by focusing on the distance function $d_N$ one obtains some general consequences.
For example, condition \eqref{eq:AlbanoCannarsa} holds for $u=d_N$ for any $x_0\in\Sigma_2(N)$
since in this case the reachable gradient $D^*d_N(x_0)$ consists of exactly two points, and hence the set $\partial d_N(x_0)=\bd[\partial d_N(x_0)]$ is a segment.
Therefore, for any point in $\Sigma_2(N)$ the above Lipschitz propagation occurs with dimension $\nu=m-1$ (see also \cite[Section 6]{Albano1999}).
We mention that condition \eqref{eq:AlbanoCannarsa} is also known to be necessary for propagation in a certain case, cf.\ \cite[Theorem 2]{Albano2002}.

It is now clarified that our Lemma \ref{lem:biLipschitzpropagate} is quite similar in spirit to Albano--Cannarsa's Lipschitz propagation result (with dimension $\nu=m-1$) and this is why we call our result ``DC graphical propagation'' in the introduction.

Finally we point out two important remarks: 
\begin{itemize}
  \item There is a (semi-)concave function $u:\R^m\to\R$ whose singular set has fractional Hausdorff dimension between $m-2$ and $m-1$ (Example \ref{ex:cantor}).
  \item There is a (semi-)concave function $u:\R^m\to\R$ with a singular point at which a Lipschitz propagation occurs but any graphical propagation does not (Example \ref{ex:zigzag}).
\end{itemize}
The first remark ensures that Theorem \ref{thm:main_singular_set} fails for a general (semi-)concave function, since if the same assertion would hold, then we would necessarily have either $\dim_\mathcal{H}\Sigma(u)=m-1$ or $\dim_\mathcal{H}\Sigma(u)\leq m-2$.
The second remark certifies a significant difference between Lipschitz and (DC) graphical propagations at least for general (semi-)concave functions.

\section{Preliminaries on Finsler geometry}\label{sec:PreliminariesFinsler}

In what follows we argue on a general Finsler manifold $M$.
In this section we review basic definitions, notions and properties in the framework of Finsler geometry, while introducing terminologies that we use in the sequel.

\subsection{Lipschitz submanifold}\label{subsec:Lipsubmfd_Finsler}

To begin with, we give a definition of Lipschitz submanifold in a general smooth manifold (independent of the metric).

\begin{definition}[Lipschitz submanifold: General case]\label{def:Lipsubmfd_Finsler}
  Let $M^m$ be a smooth $m$-dimensional manifold and $d$ an integer with $1\leq d\leq m-1$.
  A subset $S\subset M^m$ is called \emph{(embedded) $d$-dimensional Lipschitz submanifold} if for any $p\in S$ there is a local chart $\varphi:U\to\varphi(U)\subset\R^m$ around $p$ such that $\varphi(S\cap U)$ is a $d$-dimensional Lipschitz submanifold of $\R^m$ in the sense of Definition \ref{def:Lipsubmfd_Euclid}.
\end{definition}

In fact, the above definition does not depend on the choice of a local chart; in particular this definition is compatible with Definition \ref{def:Lipsubmfd_Euclid} provided that $M$ is Euclidean.
Although this fact might be well known for experts, we could not find any explicit argument in the literature so here we verify the following

\begin{lemma}\label{lem:submanifold_coordinate}
  Let $S\subset\R^m$ be a subset.
  Let $p\in S$ and $U\subset\R^m$ be an open neighborhood of $p$ in $\R^m$.
  Suppose that there is an affine subspace $P\subset\R^m$ such that the orthogonal projection $\pi:\R^m\to P$ restricted to $S\cap U$ is bi-Lipschitz.
  Then, for any neighborhood $V\subset\R^m$ of $p$ in $\R^m$ and any diffeomorphism $\Phi:V\to\Phi(V)\subset\R^m$, there are a neighborhood $V'\subset V\cap U$ of $p$ and an affine subspace $P'$ of same dimension as $P$ such that the orthogonal projection $\pi':\R^m\to P'$ restricted to $\Phi(S\cap V')$ is bi-Lipschitz.
\end{lemma}

\begin{proof}
  We first note that, since any orthogonal projection is clearly $1$-Lipschitz, it is sufficient to verify the estimate of the form $|\pi'(p_1)-\pi'(p_2)|\geq c|p_1-p_2|$ for any $p_1,p_2\in\Phi(S\cap V')$, where $c\in(0,1)$.

  Fix any $P$ in the assumption.
  Up to a translation we may assume that $P$ is a linear subspace, i.e., $\bo\in P$.
  Hereafter $A^\perp$ denotes the orthogonal complement of $A$.
  Fix any diffeomorphism $\Phi:V\to\Phi(V)\subset\R^m$.
  Let $T$ be a linear isomorphism $T:=d\Phi_p:\R^m\to\R^m$, and
  $$P':=\big( T(P^\perp) \big)^\perp.$$
  Clearly $P$ and $P'$ have the same dimension.
  We now prove that the linear map $\pi'\circ T|_P: P\to P'$ is injective.
  Suppose that $(\pi'\circ T|_P)(v)=\bo$ for some $v\in P$.
  Then $T(v)\in ( T(P^\perp) )^{\perp\perp} = T(P^\perp)$ and hence $v\in P^\perp$.
  Therefore, $v\in P\cap P^\perp$, i.e., $v=\bo$.
  This means that the above linear map is injective, and hence there is $c_0>0$ such that
  \begin{equation}\label{eq:0316-4}
    (\pi'\circ T|_P)(v) \geq c_0|v| \quad \text{for any $v\in P$}.
  \end{equation}
  On the other hand, since $\Phi$ is smooth, for any $\varepsilon>0$ there is $\delta>0$ such that
  $$|\Phi(q_1)-\Phi(q_2)-T(q_1-q_2)|\leq \varepsilon|q_1-q_2|$$
  for $q_1,q_2\in B_\delta(p)\cap V$, and hence
  \begin{align}\label{eq:0316-6}
    \begin{split}
      |\pi'(\Phi(q_1))-\pi'(\Phi(q_2))| &= |\pi'(\Phi(q_1)-\Phi(q_2)-T(q_1-q_2))+\pi'(T(q_1-q_2))|\\
      & \geq |\pi'(T(q_1-q_2))|-\varepsilon|q_1-q_2|.
    \end{split}
  \end{align}
  Let $\pi^\perp$ denote the orthogonal projection to $P^\perp$.
  Then $a=\pi(a)+\pi^\perp(a)$ for any $a\in\R^m$.
  Taking $a=q_1-q_2$ and operating the linear map $T$, we obtain
  $$T(q_1-q_2)=T(\pi(q_1-q_2))+T(\pi^\perp(q_1-q_2)),$$
  and hence by \eqref{eq:0316-4},
  \begin{equation}\label{eq:0316-5}
    |\pi'(T(q_1-q_2))|=|(\pi'\circ T)(\pi(q_1-q_2))|\geq c_0|\pi(q_1)-\pi(q_2)|.
  \end{equation}
  In addition, by the bi-Lipschitz property of $\pi$ there is $c_1>0$ such that
  \begin{equation}\label{eq:0316-7}
    |\pi(q_1)-\pi(q_2)|\geq c_1|q_1-q_2| \quad \text{for any $q_1,q_2\in S\cap U$}.
  \end{equation}
  Combining \eqref{eq:0316-6}, \eqref{eq:0316-5}, and \eqref{eq:0316-7}, and choosing $\varepsilon=c_0c_1/2$,
  we deduce that
  $$|\pi'(\Phi(q_1))-\pi'(\Phi(q_2))| \geq \varepsilon|q_1-q_2| \quad \text{for any $q_1,q_2\in S\cap U\cap B_\delta(p)\cap V$}.$$
  Taking $V':=U\cap B_\delta(p)\cap V$ completes the proof.
\end{proof}

With this lemma at hand we observe that if $S\subset\R^m$ is a Lipschitz submanifold in the sense of Definition \ref{def:Lipsubmfd_Finsler}, then it also satisfies Definition \ref{def:Lipsubmfd_Euclid}.
(The converse obviously follows by choosing $\varphi:=\mathrm{id}$.)
Indeed, for any $p\in S$ and the associated local coordinate $\varphi:U\to\varphi(U)\subset\R^m$ in Definition \ref{def:Lipsubmfd_Finsler}, if we apply Lemma \ref{lem:submanifold_coordinate} to the set $\varphi(S\cap U)$ with $\Phi:=\varphi^{-1}:\varphi(U)\to U$ around $\varphi(p)$,
then we find a desired bi-Lipschitz orthogonal projection of $S$ around $p$ in Definition \ref{def:Lipsubmfd_Euclid}.


\subsection{DC submanifold}

\begin{definition}[DC submanifold: General case]\label{def:DCsubmfd_Manifold}
  Let $M$ be a smooth $m$-dimensional manifold and $d$ an integer with
  $1\leq d\leq m-1.$
  A subset $S$ of  $M$ is called \emph{$d$-dimensional delta-convex (DC) submanifold} of $M$ if for any local chart $(U,\varphi),$ the set $\varphi(S\cap U)$ is a DC submanifold of $\R^m$ in the sense of Definition \ref{def:DCsubmfd_Euclid}. 
\end{definition}

As in the Euclidean space we call the codimension-one case \emph{hypersurface}, and interpret the zero-dimensional case as a point.

By the following lemma, one can say that $\varphi^{-1}(S\cap \varphi(U))$ is a DC submanifold of an $m$-dimensional smooth manifold $M$ for any local chart $(U, \varphi)$ of $M$
if a subset $S\subset \R^m$ is a DC submanifold of $\R^m$ in the sense of Definition \ref{def:DCsubmfd_Euclid}.

\begin{lemma}\label{lem:well_defined_definition_DC}
Let $S$ be a $d$-dimensional \emph{DC} submanifold of $\R^m.$
Then for any open set $V\subset\R^m$ and any diffeomorphism $\Phi: V\to \Phi(V)\subset \R^m,$ the set
$\Phi(S\cap V)$ is also a DC submanifold of $\R^m.$
\end{lemma}     
\begin{proof}
   Choose any point $\bar{p}=\Phi(p)$, where $p\in S\cap V,$ in $\Phi(S\cap V).$
   Since $S$ is a $d$-dimensional DC submanifold of $\R^m,$ there exist an open neighborhood $U$ of $p,$ $d$-dimensional affine subspace $P\subset \R^m,$
   and a DC map $\tilde{f} : \pi(S\cap U)\to \R^{m-d},$
   where $\pi$ denotes the orthogonal projection onto $P$ such that, for an isometry $T$ on $\R^m$ satisfying $T(P)=\R^d\times\{\bo \}$,
   \begin{equation*}
   T(S\cap U)=\{(x,\tilde{f}(T^{-1}(x)) \mid x\in T(\pi(S\cap U))\}.
   \end{equation*}
   Here and in the sequel each point $(x,\bo)\in \R^d\times\{\bo\}$ is identified as $x\in \R^d.$
   From Lemma \ref{lem:submanifold_coordinate} it follows that there exist an open neighborhood $V'\subset V\cap U$ of $p,$ an affine subspace $P'$ of $\R^m$ and a Lipschitz map $\tilde{g}:\pi'(\Phi(S\cap V'))\to \R^{m-d},$ where $\pi'$ denotes the orthogonal projection onto $P'$
   such that, for an isometry $T'$ of $\R^m$ satisfying $T'(P')=\R^d\times \{\bo\}$,
   \begin{equation*}
       T'(\Phi(S\cap V'))=\{(y,\tilde{g}(T'^{-1}(y)) \mid y\in T'(\pi'(\Phi(S\cap V'))\}.
   \end{equation*}
   Thus we have 
   \begin{equation}\label{eq:graph_local_S}
   S\cap V'=\{T^{-1}(x,f(x)) \mid x\in T(\pi(S\cap V'))\},
   \end{equation}
   where $f:=\tilde{f}\circ T^{-1},$
   and 
   \begin{equation}\label{eq:graph_local_Phi(S)}
       \Phi(S\cap V')=\{T'^{-1}(y,g(y)) \mid y\in T'(\pi'(\Phi(S\cap V')))\},
   \end{equation}
   where $g:=\tilde{g}\circ T'^{-1}.$
  From \eqref{eq:graph_local_S} and \eqref{eq:graph_local_Phi(S)}
  it follows that
  for each point $x\in T(\pi(S\cap V')),$ there exists a point
  $y(x)\in T'(\pi'(\Phi(S\cap V')))$ satisfying
  \begin{equation}\label{eq:change_coordinate}\nonumber
  \Phi(T^{-1}(x,f(x)) )=T'^{-1}( (y(x),g(y(x)) )).
  \end{equation}
  In particular,
  \begin{equation}\label{eq:change_coordinate1}\nonumber
  y(x)=\pi_1(T'(\Phi( T^{-1}(x,f(x)) ) ) ),
  \end{equation}
  where $\pi_1:\R^m\to \R^d\times\{\bo \}$ denotes the orthogonal projection onto $\R^d\times\{\bo \},$ and
  \begin{equation}\label{eq:change_coordinate2}
  g(y(x))=\pi_2(T'(\Phi(T^{-1}(x,f(x))) ) ),
  \end{equation}
  where $\pi_2:\R^m\to \{\bo \}\times\R^{m-d}$ denotes the orthogonal projection onto $\{\bo \}\times\R^{m-d}.$
  The map $y=y(x)$ is bi-Lipschitz and locally DC on $T(\pi(S\cap V')))$ by Lemma \ref{lem:DC_fundamental_properties} (ii), (iii).
  Therefore, from \cite[Theorem 5.2]{Vesely1989} it follows that the inverse map $x=x(y)$ of $y=y(x)$
  is also DC on $T'(\pi'(\Phi(S\cap V'))).$ Moreover, from \eqref{eq:change_coordinate2} we get 
  \begin{equation*}
      g(y)=\pi_2(T'(\Phi(T^{-1}(x(y),f(x(y)))))).
  \end{equation*}
  Now it is clear from Lemma \ref{lem:DC_fundamental_properties} (ii) that
  $g$ is locally DC. This implies that $\Phi(S\cap V)$ is a DC submanifold of $\R^m.$
\end{proof} 

\subsection{Generalized differential}

Let us introduce here the notion of the {\it generalized differential} of a locally Lipschitz function on a smooth manifold (again independent of the metric).

Let $\tilde f:U\to {\bf R}$ denote a locally Lipschitz function on an open subset $U$ of ${\bf R}^m.$
Define the
{\it reachable differential} $d^*{\tilde f}(x)$
of $\tilde f$ at $x\in U$ by
\begin{equation*}
d^*\tilde f(x):=\Big\{ \lim_{j\to\infty}d\tilde f_{x_j} \ \Big|\ x_j\to x,\ \exists d\tilde f_{x_j} \Big\},
\end{equation*}
where $d\tilde f_{x_j}$ denotes the (classical) differential of $\tilde f$ at $x_j.$
Notice that the generalized gradient $\partial \tilde f(x)$ defined in \eqref{eq:generalized_gradient} is linearly isomorphic to $\conv d^*\tilde f(x).$

Then we define the {\it reachable differential} $d^*f(p)$ at a point $p\in V$ of a locally Lipschitz function $f$ on an open subset $V$ of a smooth manifold $M^m$ by
\begin{equation*}
d^*f(p):=\Big\{ \omega \circ d\varphi_p\ \Big|\ \omega\in d^*(f\circ \varphi^{-1})(\varphi(p))\Big\},
\end{equation*}
where $\varphi$ denotes a local chart around the point $p.$
From the chain rule it is clear that this definition is independent of the choice of a local chart around $p.$
It is also clear that an equivalent definition is:
\begin{equation*}
    d^*f(p):=\Big\{ \lim_{j\to\infty} df_{p_j} \Big|\ p_j\to p,\ \exists df_{p_j} \Big\}.
\end{equation*}
Here $df_{p_j}$ means the classical differential of $f$ at $p_j$ and the limit $df_{p_j}\to df_p$ means $(p_j,df_{p_j})\to(p,df_p)$ in $TM$.

Finally we define the {\it generalized differential} $\partial^* f(p)$ by
\begin{equation*}
\partial^*f(p):=\conv d^*f(p).
\end{equation*}
By definition one has similar properties to the Euclidean case; $\partial^*(f_1+f_2)\subset\partial^*f_1+\partial^*f_2$; if $\partial^*f(p)$ is a singleton, then $f$ is differentiable at $p$.

\subsection{DC implicit function theorem}

We can extend the DC implicit function theorem on the Euclidean space to a general smooth manifold.
To state it we need to define DC functions on a manifold.
Incidentally we also define semi-concavity for later use.

\begin{definition}[Delta-convexity and semi-concavity on a manifold]\label{def:SC_DC_mfd}
  A continuous function $u:\Omega\to\mathbf{R}$ on an open set $\Omega\subset M^m$ is called {\em locally delta-convex (DC)} (resp.\ {\em locally semi-concave}) if for any local chart $\varphi:\Omega\supset U\to\varphi(U)\subset\R^m$ the function $u\circ\varphi^{-1}$ is locally DC (resp.\ locally semi-concave).
\end{definition}

This definition is well-posed thanks to independence on the choice of the local chart, cf.\ \cite[Proposition 2.6]{Mantegazza2003} and Lemma \ref{lem:DC_fundamental_properties} (ii), (iii) (see also \cite{Petrunin2007}, \cite[Section 2.2]{Generau2022} for semi-convex functions on curved spaces).
Of course, any semi-concave function on a manifold is also delta-convex.

Using the definitions prepared above, we obtain the following

\begin{theorem}[DC implicit function theorem on a manifold]\label{thm:DC_IFT_mfd}
  Let $V\subset M^m$ be an open subset of an $m$-dimensional smooth manifold $M^m$ and $f_1,\dots,f_k:V\to\R$ be DC functions, where $1\leq k\leq m-1$, and let $p_0\in V$.
  If $f_1(p_0)=\dots=f_k(p_0)=0$ and 
  if $\partial^* f_1(p_0),\dots,\partial^* f_k(p_0)$ are singletons whose elements are linearly independent,
  then there exists a neighborhood $U\subset V$ of $p_0$ in $M^m$ such that $S:=\bigcap_{j=1}^k f_j^{-1}(0)\cap U$ is an $(m-k)$-dimensional DC submanifold of $M^m$.
\end{theorem}

The proof is reduced to Theorem \ref{thm:DC_IFT} by taking a local chart so safely omitted.

\subsection{Finsler manifold}

From this subsection we associate a metric to a smooth manifold.

A smooth ($C^\infty$) manifold $M$ is called {\em Finsler manifold} if the manifold admits a nonnegative function $F$
on the tangent bundle $TM$ of $M$ such that $F$ is $C^\infty$ on $TM\setminus\{\mathbf{0}\}$, where $\mathbf{0}$ denotes the zero section, and the restriction $F|_{T_pM}$ of $F$ to each tangent plane $T_pM$ to $M$ at $p\in M$ is a Minkowski norm.
Here, $F$ is called a {\em Minkowski norm} if
\begin{itemize}
  \item[1.] $F(p,y)\geq0$ for all $(p,y)\in T_pM$, and $F(p,y)=0$ if and only if $y=0$;
  \item[2.] $F$ is positively homogeneous of degree $1$, i.e., $F(p,\lambda y)=\lambda F(p,y)$ for all  $(p,y)\in T_pM$ and $\lambda\geq0$;
  \item[3.] the Hessian of $F^2$ is positive definite on $T_pM\setminus\{0\}$, i.e., for any $y\in T_pM\setminus\{0\}$, the following bilinear symmetric function $g_y$ on $T_pM$ is an inner product,
  \begin{equation}\label{eq01}\nonumber
    g_y(u,v):= \frac{\partial^2}{\partial s\partial t}\Big|_{s=t=0} \left( \frac{1}{2}F^2(p,y+su+tv) \right).
  \end{equation}
\end{itemize}
On a Finsler manifold $(M,F)$, one can define the length of curves as follows.
For a $C^1$-curve $\gamma:[a,b]\to M$, the length $L(\gamma)$ of $\gamma$ is defined by
\begin{equation}\label{eq02}
  L(\gamma):=\int_a^b F(\gamma(t),\gamma'(t))dt,
\end{equation}
where $\gamma'(t)$ denotes the velocity vector of $\gamma$.

One can introduce the Finsler distance on a $C^\infty$-Finsler manifold $(M,F)$.
For each pair of points $p,q\in M$, let $\Omega_{pq}$ denote the set of all piecewise $C^1$-curve $\gamma:[a,b]\to M$ such that $\gamma(a)=p$ and $\gamma(b)=q$.
The distance $d(p,q)$ from $p$ to $q$ is defined by
\begin{equation}\label{eq05}\nonumber
  d(p,q):=\inf_{\gamma\in\Omega_{pq}}L(\gamma).
\end{equation}
Note that the length of a piecewise $C^1$-curve is also defined by \eqref{eq02}.
This distance has the following (quasimetric) properties:
\begin{itemize}
  \item[1.] $d(p,q)\geq0$, with equality if and only if $p=q$;
  \item[2.] $d(p,q)\leq d(p,r)+d(r,q)$ holds for any points $p$, $q$, and $r$.
\end{itemize}
Notice carefully that the relation $d$ may not be symmetric due to the anisotropy of the Finsler metric.
To avoid the possible asymmetry, we also use the symmetrized distance function
$$d_{\max}(p,q):=\max\{d(p,q),d(q,p)\}.$$
This $d_{\max}$ defines a distance function in the standard sense (with symmetry).
Note that the topology induced from this distance coincides with the original topology of the manifold $M$.

We recall (cf.\ \cite[p.151]{Bao2000}) that a sequence $\{p_i\}_{i=1}^\infty$ of points on a Finsler manifold $(M,F)$ is called {\em backward} (resp.\ {\em forward}) {\em Cauchy sequence} if for any $\varepsilon>0$ there exists $N(\varepsilon)>0$ such that for all $N(\varepsilon)\leq i\leq j$, $d(p_j,p_i)<\varepsilon$ (resp.\ $d(p_i,p_j)<\varepsilon$).
The Finsler manifold $(M,F)$ is said to be {\em backward} (resp.\ {\em forward}) {\em complete} if every backward (resp.\ forward) Cauchy sequence converges.
From the Finslerian version of the Hopf-Rinow theorem (see \cite[p.168]{Bao2000}), the Finsler manifold $(M,F)$ is backward (resp.\ forward) complete if and only if every closed and backward (resp.\ forward) bounded subset of $(M,F)$ is compact, where a subset $N$ of $(M,F)$ is said to be {\em backward} (resp.\ {\em forward}) {\em bounded} if there exists a point $p\in M$ and a number $K>0$ such that $\sup_{x\in N}d(x,p)<\infty$ (resp.\ $\sup_{x\in N}d(p,x)<\infty$).
If a Finsler manifold is backward and forward complete, the manifold is said to be {\em bi-complete}, or simply \emph{complete} in this paper.

Let $N\subset M$ be a closed subset and $d_N$ be the distance function from $N$ defined in \eqref{eq:distance_def}.
A curve $\gamma:[0,a]\to M$ is called {\em $N$-segment} if $d(N,\gamma(t))=t$ holds for any $t\in[0,a]$.
By this definition with continuity of $d_N$ any curve obtained as a pointwise limit of a sequence of $N$-segments is also an $N$-segment.

Any $N$-segment $\gamma$ is a unit-speed geodesic.
Here a geodesic is defined as a solution to the Euler--Lagrange equation corresponding to the first variation of the length functional, cf.\ \cite[Ch.\ 5]{Bao2000}, and hence any geodesic is a unique smooth solution to a second order ODE for given initial data of up to first order.
Based on this fact, we can prove (in a standard way) that any two $N$-segments do not intersect with each other except at their endpoints.
Also, by using the exponential map
(defined on the whole $TM$ by the Hopf--Rinow theorem \cite[Theorem 6.6.1]{Bao2000}) the curve $\gamma$ can be written of the form $\exp_p(tT)$, where $p=\gamma(0)$ and $T=\gamma'(0)\in S_pM$.
We remark that, in contrast to the Riemannian case, the exponential map is only $C^1$ at the zero section of $TM$ but $C^\infty$ away from it \cite[p.127]{Bao2000}.

\subsection{Basic properties of distance functions}

For later use we exhibit several basic properties of the distance function $d_N:M\setminus N\to(0,\infty)$ from a nonempty closed subset $N$ of a connected complete smooth Finsler manifold $M$.

The singular set $\Sigma(N)$ is already defined in \eqref{eq:def_singular_set} as the set of nondifferentiability points of $d_N$.
This set can also be geometrically characterized even in the Finsler case: A point $p\in M\setminus N$ is not an element of $\Sigma(N)$ if and only if $p$ admits a unique $N$-segment \cite[Theorem 2.3]{Sabau2016}.
This fact provides the following characterization:
$$\Sigma(N)=\big\{ p\in M\setminus N \mid \#\pi_{UN}(p)\geq2 \big\},$$
where $UN\subset TM$ denotes the (formal) unit normal bundle of $N$ defined by
\begin{equation*}
  UN:=\big\{ (\alpha(0),\alpha'(0)) \in TM \mid \text{$\alpha$ is an $N$-segment} \big\},
\end{equation*}
and $\pi_{UN}:M\setminus N \to 2^{UN}$ denotes the projection map defined by
\begin{equation*}
  \pi_{UN}(p):= \big\{(\alpha(0),\alpha'(0))\in UN \mid \text{$\alpha$ is an $N$-segment to $p$} \big\}.
\end{equation*}
The map $\pi_{UN}$ is also set-valued upper semicontinuous; indeed, if $p_j\to p$ in $M\setminus N$ and $(q_j,T_j)\in\pi_{UN}(p_j)$ converges to some $(q,T)\in UN$, then the $N$-segments $\alpha_j(t)=\exp_{q_j}(tT_j)$ converge to $\alpha(t):=\exp_{q}(tT)$, which also defines an $N$-segment to $p$ so that $(q,T)\in\pi_{UN}(p)$.
Notice the difference from the Euclidean case that the image of the projection consists of not only points in $M$ but also tangent directions.
This is because in the Finslerian (or even Riemannian) case, common endpoints may admit non-unique $N$-segments (but in this case the tangent vectors need to be different at the initial endpoint).
Finally, let $\pi_N:M\setminus N \to 2^N$ be the composition of $\pi_{UN}$ and the canonical projection from $UN$ onto $N$, namely
\begin{equation*}
  \pi_N(p):= \{\alpha(0)\in N \mid \text{$\alpha$ is an $N$-segment to $p$}\}.
\end{equation*}
Clearly, this map is also set-valued upper semicontinuous.

We also have a similar characterization of $d^*d_N$ as in the Euclidean case.
\begin{lemma}
    The reachable differential $d^*d_N$ of $d_N$ at $p\in M\setminus N$ is given by
    \begin{equation}\label{eq:distance_differential}
         d^*d_N(p)=\left \{g_v(v,\cdot)\ \big|\ v=\alpha'(d_N(p)),\ \text{$\alpha$ is an $N$-segment to  $p$} \right\}.
    \end{equation}
\end{lemma}
\begin{proof}
    We confirm that the l.h.s.\ is contained in the r.h.s.
    For any element $\omega\in d^*d_N(p)$ there is a sequence $p_j\to p$ in $M\setminus N$ such that $\omega=\lim_{j\to\infty}dd_N(p_j)$.
    By \cite[Theorem 2.3]{Sabau2016} there is a unique $N$-segment $\alpha_j(t)=\exp_{q_j}(tT_j)$ to $p_j$ and $dd_N(p_j)=g_{v_j}(v_j,\cdot)$, where $v_j:=\alpha_j'(d_N(p_j))$.
    By compactness we can take a subsequence of $\{(q_j,T_j)\}$ (without relabeling) so that $(q,T):=\lim_{j\to\infty}(q_j,T_j)$ exists.
    Then the curve $\alpha(t):=\exp_p(tT)$ is an $N$-segment to $p$, and we have 
    $$v:=\alpha'(d_N(p))=\lim_{j\to\infty}\alpha_j'(d_N(p_j))=\lim_{j\to\infty}v_j$$
    since the exponential map is smooth in this limit (or, unique solutions to the geodesic equation smoothly depend on initial data). 
    Hence $\omega=\lim_{j\to\infty}g_{v_j}(v_j,\cdot)=g_v(v,\cdot)$.
    The converse inclusion is easier to check since for any $N$-segment $\alpha$ to $p$, by \cite[Theorem 2.3]{Sabau2016} we can differentiate $d_N$ at $\alpha(d_N(p)-1/j)$ with differential $g_{v_j}(v_j,\cdot)$, where $v_j:=\alpha'(d_N(p)-1/j)$ which converges to $\alpha'(d_N(p))$ as $j\to\infty$.
\end{proof}

In particular, $d_N$ is differentiable at $p$ if and only if $d^*d_N(p)$ or $\partial^*d_N(p)$ is a singleton.

\section{Generic DC structure: Finslerian case}\label{sec:StructureFinslerian}

Here we complete the proof of Theorem \ref{thm:main_singular_set}.
All the key ideas and techniques are already given in Section \ref{sec:StructureEuclidean}.
This section is thus mainly devoted to explaining how to extend them to the general Finslerian case.

The set $\Sigma_2(N)$ is defined similarly to \eqref{eq:Sigma_2}, namely
\begin{equation*}
  \Sigma_2(N):=\{p\in\Sigma(N) \mid \#\pi_{UN}(p)=2 \}.
\end{equation*}
Then our precise goal is as follows:

\begin{theorem}[DC hypersurface structure: Finslerian case]\label{thm:Finslerian_singular_set}
  Let $m\geq2$, $M^m$ be an $m$-dimensional smooth complete connected Finsler manifold, and $N\subset M^m$ be a nonempty closed subset.
  Then there exist at most countably many DC hypersurfaces $S_j\subset M$ and $(m-2)$-dimensional DC submanifolds $S'_j\subset M$ such that
  \begin{itemize}
    \item[(i)] $\Sigma_2(N) \subset S \subset \Sigma(N)$, where $S:=\bigcup_{j=1}^\infty S_j$,
    \item[(ii)] $\Sigma(N)\setminus S\subset\bigcup_{j=1}^\infty S'_j$.
  \end{itemize}
\end{theorem}

For each $p\in\Sigma(N)$ we define the radius function analogously to \eqref{eq:rad}:
\begin{equation}\label{eq:rad_mfd}
  \rad_N(p):=\max\{ d_{\max}(q_1,q_2) \mid q_1,q_2\in\pi_N(p) \}.
\end{equation}
Note that, in contrast to the Euclidean case, the radius may take zero since a single point of $N$ may admit multiple $N$-segments to $p$.
This issue is however not essential by considering a closed tubular neighborhood of $N$ as in the following lemma, the proof of which is safely omitted.

\begin{lemma}\label{lem:closedneighborhood}
  Let $K\subset M\setminus N$ be a compact subset.
  Let $c\in(0,\min_{x\in K}d_N(x))$ and $N_c:=\{x\in M \mid d_N(x)\leq c \}$.
  Then $d_{N_c}(p)+c=d_N(p)$ holds for any $p\in K$.
  In particular, $\Sigma_2(N)\cap K=\Sigma_2(N_c)\cap K$ and $\rad_{N_c}(p)>0$ for $p\in K$.
\end{lemma}

Now we confirm that $d_N$ is locally DC, or more strongly, locally semi-concave in the sense of Definition \ref{def:SC_DC_mfd}.
This fact is proven in \cite[Proposition 3.4]{Mantegazza2003} in the Riemannian case through viscosity solution theory.
Their proof is based on \cite[Theorem 5.3]{Lions1982}, which might extend to some Finsler setting.
Here instead, we give a different and direct proof of the general Finsler case.

\begin{lemma}\label{lem:semiconvexity}
  The function $d_N$ is locally semi-concave (and hence DC) on $M\setminus N$.
\end{lemma}

\begin{proof}
  Let $\varphi:M\setminus N\supset U\to\varphi(U)\subset\R^m$ be any local chart.
  We prove that
  $$u:=d_N\circ\varphi^{-1}$$
  is locally semi-concave.
  By locality it is sufficient to prove that for any $z\in \varphi(U)$ there is an open ball $B_\delta(z)$ centered at $z$ such that $\overline{B_\delta(z)}\subset \varphi(U)$ and $d_N\circ\varphi^{-1}$ is semi-concave on $B_\delta(z)$.
  Fix any $z_0\in \varphi(U)$ and take a small radius $\delta_0>0$ so that $B_{2\delta_0}(z_0)\subset \varphi(U)$ and the function $d(\varphi^{-1}(\cdot),\varphi^{-1}(\cdot)):\varphi(U)\times \varphi(U)\to[0,\infty)$ is smooth on $B_{2\delta_0}(z_0)\times B_{2\delta_0}(z_0)\setminus\Delta$, where $\Delta$ denotes the diagonal set.
  (Such a small radius exists thanks to \cite[Theorem 3.1]{Shiohama2019}.)
  Fix any $x_0,x_1\in B_{\delta_0}(z_0)$ and $\lambda\in[0,1]$, and let $x_\lambda:=\lambda x_0+(1-\lambda)x_1$.
  It now suffices to verify that there is $C>0$ depending only on $\varphi$, $z_0$, $\delta_0$ such that
  \begin{align}\label{eq:0316-1}
    \lambda u(x_0) + (1-\lambda) u(x_1) - u(x_\lambda) \leq C \lambda(1-\lambda)|x_0-x_1|^2.
  \end{align}
  Let $p_i:=\varphi^{-1}(x_i)$ for $i=0,\lambda,1$.
  Let $\alpha_\lambda:[0,d_N(p_\lambda)]\to M$ be an $N$-segment to $p_\lambda$, and choose $t_0\in[0,d_N(p_\lambda)]$ so that $\alpha_\lambda(t_0)\in U$ and that the point $y_\lambda:=\varphi(\alpha_\lambda(t_0))$ satisfies that $|z_0-y_\lambda|=\frac{3}{2}\delta_0$.
  By minimality of $\alpha_\lambda$,
  \begin{align*}
    u(x_\lambda)=d_N(p_\lambda) = d_N(\alpha_\lambda(t_0)) + d(\alpha_\lambda(t_0),p_\lambda) = u(y_\lambda) + d(\varphi^{-1}(y_\lambda),\varphi^{-1}(x_\lambda)),
  \end{align*}
  while for $i=0,1$, by the triangle inequality,
  \begin{align*}
    u(x_i)=d_N(p_i) \leq d_N(\alpha_\lambda(t_0)) + d(\alpha_\lambda(t_0),p_i) = u(y_\lambda) + d(\varphi^{-1}(y_\lambda),\varphi^{-1}(x_i)).
  \end{align*}
  Using these relations, and letting
  $v:=d(\varphi^{-1}(y_\lambda),\varphi^{-1}(\cdot)),$
  we deduce that
  \begin{align}\label{eq:0316-2}
    \lambda u(x_0) + (1-\lambda) u(x_1) - u(x_\lambda) \leq \lambda v(x_0) + (1-\lambda) v(x_1) - v(x_\lambda).
  \end{align}
  On the other hand, by definition of $\delta_0$ and $|z_0-y_\lambda|=\frac{3}{2}\delta_0$, the function $v$ is smooth on $B_{\frac{3}{2}\delta_0}(z_0)$, and hence $|D^2v|\leq C$ holds on $B_{\delta_0}(z_0)$ for a constant $C>0$ depending only on ($M$ and) $\varphi,z_0,\delta_0$.
  Therefore,
  \begin{align}\label{eq:0316-3}
    \lambda v(x_0) + (1-\lambda) v(x_1) - v(x_\lambda) \leq C\lambda(1-\lambda)|x_0-x_1|^2.
  \end{align}
  Estimates \eqref{eq:0316-2} and \eqref{eq:0316-3} imply \eqref{eq:0316-1}.
\end{proof}

We then obtain the main propagation lemma, a Finslerian version of Lemma \ref{lem:biLipschitzpropagate}.
Although the statement and the proof are similar except for terminological differences, we give a somewhat precise argument for clarity.

\begin{lemma}\label{lem:biLipschitzpropagate_Finslerian}
  Let $p\in \Sigma_2(N)$ such that $\rad_N(p)>0$.
  Then there exist a positive number $\delta(p)>0$ and a DC hypersurface $S_p\subset M$ such that $S_p\subset\Sigma(N)$ and such that $\rad_N(y)\leq \frac{1}{2}\rad_N(p)$ holds for any $y\in (\Sigma(N)\cap B_{\delta(p)}(p))\setminus S_p$, where $B_r(p):=\{x\in M \mid d_\mathrm{max}(p,x)<r\}$.
\end{lemma}

\begin{proof}
  Let $\pi_N(p)=\{q_1,q_2\}$.
  Then $q_1\neq q_2$ by $\rad_N(p)>0$.
  Let $N_j:=N\cap\{ y\in M \mid d_{\max}(q_j,y)\leq \frac{1}{4}\rad_N(p) \}$.
  The sets $N_1$ and $N_2$ are compact and, since $\rad_N(p)>0$, they are (relative) neighborhoods of $q_1$ and $q_2$ in $N$.
  In addition, by $\rad_N(p)=d_{\max}(q_1,q_2)$, we have
  \begin{equation}\label{eq20}
    N_1\cap N_2=\emptyset,
  \end{equation}
  since otherwise there would be $y\in N_1\cap N_2$ but then $\rad_N(p)=d_{\max}(q_1,q_2)\leq d_{\max}(q_1,y)+d_{\max}(y,q_2)\leq \frac{1}{2}\rad_N(p)$, which contradicts $\rad_N(p)>0$.

  Let $f:=d_{N_1}-d_{N_2}=d(N_1,\cdot)-d(N_2,\cdot)$.
  Since $\pi_N$ is set-valued upper semicontinuous, there is $r>0$ such that for any $x\in B_r(p)$ we have $\pi_N(x)\subset N_1\cup N_2$, and in particular $d_N(x)=\min\{d_{N_1}(x),d_{N_2}(x)\}$.
  Hence, if $x\in B_r(p)$ and $f(x)=0$, then $d_{N_1}(x)=d_{N_2}(x)=d_N(x)$ and hence, by \eqref{eq20}, $\#\pi_N(x)\geq2$.
  Therefore, for any $\delta\in(0,r]$,
  \begin{equation*}
    S_{p,\delta} := f^{-1}(0)\cap B_\delta(p)\subset \Sigma(N).
  \end{equation*}
  In addition, if $x\in(\Sigma(N)\cap B_\delta(p))\setminus S_{p,\delta}$, then either $\pi_N(x)\subset N_1$ or $\pi_N(x)\subset N_2$ holds (depending on the sign of $f(x)$) and hence $\rad_N(x)\leq\frac{1}{2}\rad_N(p)$.

  We finally prove that for a suitable $\delta:=\delta(p)\in(0,r]$ the set $S_p:=S_{p,\delta(p)}$ is a DC hypersurface.
  Notice that $f$ is DC around $p$, cf.\ Lemma \ref{lem:semiconvexity} and Lemma \ref{lem:DC_fundamental_properties} (i).
  Since $p$ admits a unique $N_j$-segment $\alpha_j$ from $q_j$ for $j=1,2$, letting $v_j:=\alpha_j'(d_N(p))$ and $\omega_j:=g_{v_j}(v_j,\cdot)$, we have $\partial^*d_{N_j}(p)=d^* d_{N_j}(p)=\{\omega_j\}$ by \eqref{eq:distance_differential}.
  Hence we obtain $\partial^* f(p)=\{\omega_1-\omega_2\}$ (as in the proof of Lemma \ref{lem:biLipschitzpropagate}).
  Since $\omega_1(v_1)=1>\omega_2(v_1)$ by \cite[(1.2.4), (1.2.16)]{Bao2000}, we have $\omega_1\neq \omega_2$.
  Thanks to this fact with $f(p)=0$, we now deduce by Theorem \ref{thm:DC_IFT_mfd} (with $k=1$) that there is a small positive number $\delta(p)>0$ such that $S_p=f^{-1}(0)\cap B_{\delta(p)}(p)$ is a Lipschitz hypersurface.
\end{proof}

We are now in a position to complete the proof of Theorem \ref{thm:main_singular_set}, mostly appealing to the proof of Theorem \ref{thm:Euclidean_singular_set}.

\begin{proof}[Proof of Theorem \ref{thm:Finslerian_singular_set}]
  We mainly follow the proof of Theorem \ref{thm:Euclidean_singular_set} with $\mathbf{R}^m$ replaced by $M$, and in particular Lemma \ref{lem:biLipschitzpropagate} by Lemma \ref{lem:biLipschitzpropagate_Finslerian}.
  In what follows we only emphasize the differences.

  Concerning Step 1, the only difference we need to be careful is that for a compact set $K\subset M$, the set $\Sigma_2(N)\cap K$ may not be equal to $\bigcup_{i=1}^{\infty} A_i$ since $\{p\in\Sigma_2(N)\cap K\mid\rad_N(p)=0\}$ may not be empty, where $A_i$ is defined in the same way as in \eqref{eq:A_i} by using $\rad_N(K)$ defined by \eqref{eq:rad_K} with $\rad_N(p)$ interpreted as \ref{eq:rad_mfd}.
  To overcome this difference, it is sufficient to just choose a constant $c>0$ as in Lemma \ref{lem:closedneighborhood} (depending on $K$) and to replace $N$ by $N_c$.
  Then the completely parallel argument implies that for any compact set $K\subset M$ the set $\Sigma_2(N_c)\cap K =\Sigma_2(N)\cap K$ is covered by an at most countably many DC hypersurfaces.
  Taking an increasing sequence of $K$ completes the proof.

  Concerning Step 2, there is no essential difference but we need to interpret some objects through differentials, so we give a precise argument.
  Choose any $p\in\Sigma(N)\setminus\Sigma_2(N).$
  There exist at least three distinct $N$-segments $\alpha_j:[0,L]\to M$ such that $\alpha_j(L)=p$, where $j=1,2,3$ and $L:=d_N(p)$, and there are the 1-forms $\omega_j(\cdot):=g_{v_j}(v_j,\cdot)$ corresponding to the unit vectors $v_j:=\alpha_j'(L)\in T_pM$, $j=1,2,3$.
  Now we verify that
  \begin{equation}\label{eq:0325_01}
    \text{the three 1-forms $\omega_1,\omega_2,\omega_3$ are not colinear.}
  \end{equation}
  In fact, suppose that they are colinear. We may assume that $\omega_1$ is between $\omega_2$ and $\omega_3,$ i.e., there exists a $\lambda\in(0,1)$ satisfying $\omega_1=\lambda\omega_2+(1-\lambda)\omega_3.$
  Since $\omega_1(v_1)=1$, $\omega_2(v_1)<1$, $\omega_3(v_1)<1$ by \cite[(1.2.4), (1.2.16)]{Bao2000},
  we obtain
  \begin{equation*}
  1=\omega_1(v_1)=\lambda\omega_2(v_1)+(1-\lambda)\omega_3(v_1)<1,
  \end{equation*}
  which is a contradiction.
  Hence \eqref{eq:0325_01} holds true.
  This means that
  \begin{equation*}
    \text{$\Sigma(N)\setminus\Sigma_2(N)\subset \Sigma^2(d_N)$, where $\Sigma^{k}(d_N):=\{ p\in M \mid \dim \partial^*d_N(p)\geq k\}$.}
  \end{equation*}
  Thanks to Lemma \ref{lem:semiconvexity},
  the function $d_N\circ\varphi^{-1}$ is locally semi-concave and hence locally DC for any local coordinate system $(U,\varphi)$.
  Therefore, by \cite[Theorem 8]{Zajicek1983}, the set
  $\Sigma^2(d_N\circ\varphi^{-1})$ is covered by countably many $(m-2)$-dimensional DC hypersurfaces in $\R^m$.
  By Definition \ref{def:DCsubmfd_Manifold} and Lemma \ref{lem:well_defined_definition_DC},
  the set
  $\Sigma^{2}(d_N)$ is also covered by countably many $(m-2)$-dimensional DC hypersurfaces in $M$.
\end{proof}

\begin{proof}[Proof of Theorem \ref{thm:main_singular_set}]
  It is now a direct consequence of Theorem \ref{thm:Finslerian_singular_set}.
\end{proof}

\begin{remark}[$C^2$-rectifiability]
    As already mentioned, by a covering argument one can even deduce that the set $R$ in Theorem \ref{thm:main_singular_set} is covered by a countable family of $C^2$-hypersurfaces up to a negligible set with respect to the $\mathcal{H}^{m-2}$-measure.
\end{remark}

\section{Fine structure in dimension two}\label{sec:Structure2D} 

In this final section we prove that if the ambient manifold is two dimensional, then the exceptional set in Theorem \ref{thm:Finslerian_singular_set} is also covered by a countable family of DC Jordan arcs, thus proving Theorem \ref{thm:2D_singular_set}.
Here and hereafter a Jordan arc means a continuous injective curve.

\begin{definition}[DC Jordan arc]\label{def:DC_Jordan_arc}
    A subset $S\subset M^m$ is called \emph{DC Jordan arc} if $S=c([a,b])$ holds for some Jordan arc $c:(a-\epsilon,b+\epsilon)\to M^m$ with $\epsilon>0$ such that $c((a-\epsilon,b+\epsilon))$ is a one-dimensional DC submanifold.
\end{definition}

\subsection{Sectors}

A key tool here is a sector from differential geometry.
The concept of a sector at a cut point was first introduced in
\cite{shiohama1993}.
The sector is a very powerful tool and has been used to show various structures and properties of the cut locus of a smooth Jordan curve in a 2-dimensional complete Riemannian manifold
(see e.g.\ 
\cite[Theorem (Fiala--Hartman--Shiohama--Tanaka)]{Itoh2001}, \cite[Propositions 4.2.2, 4.2.3]{SST2003}).
In \cite{Sabau2016}, the detailed structure of the cut locus of $N$ was studied by making use of sectors.
Our investigation below is also in such a direction.

To begin with, we need to review convex balls before  introducing a sector at a cut point.
Let $M$ denote any dimensional Finsler manifold which is not always complete. For each point $p\in M$, let $B_\delta(p)$ denote the forward or backward open ball centered at $p$ with radius 
$\delta>0$. 
If $\delta$ is sufficiently small, then $B_\delta(p)$ has the following property:

{\it For any pair of points $x,y$ in the closure $\overline{B_\delta(p)}$,
there exists a unique minimal geodesic segment $\gamma:[a,b]\to M$ joining $x$ to $y$ with $\gamma((a,b))\subset B_\delta(p)$.}

If a ball $B_\delta(p)$ satisfies the above property,  then the ball is said to be \textit{strongly convex.} 
In fact, for each point $p$ the function $d^2(p,\gamma(t))$ (resp.\ $d^2(\gamma(t),p))$ is a strongly convex function for any unit speed geodesic segment $\gamma:[a,b]\to B_\delta(p)$ if $B_\delta(p)$ is a forward (resp.\ backward) ball and if $\delta$ is sufficiently small (cf.\ \cite [Section 2.8, Whitehead's Theorem]{Postnikov1967}).
Thus, if $B_{\delta_0}(p)$ is strongly convex, then $B_\delta
(p)$ is also strongly convex for all $\delta\in(0,\delta_0)$.
The existence of a strongly convex ball at each point is well known in Riemannian Geometry, and proven in \cite{Shiohama2019} in a Finsler manifold.

From now on, $N$ denotes a closed subset of a complete 2-dimensional Finsler manifold $M$.
Let $p\in\Sigma(N)$, and choose a small number $\delta_0>0$ such that the forward ball
$B_{\delta_0}(p):=\{q\in M \mid d(p,q)<\delta_0\}$ is strongly convex and $\overline{B_{\delta_0}(p)}\subset M\setminus N$.
Let us introduce a sector at $p$.

\begin{definition}[Sector]
Each connected component of 
$B_{\delta_0}(p)
\setminus 
\bigcup_{\gamma\in \Gamma_p}\gamma([0,d_N(p)])$,
where $\Gamma_p$ denotes the set of all $N$-segments to $p$, is called \textit{sector} at $p$.
\end{definition}

\begin{remark}
We collect some elementary properties of a sector.
\begin{itemize}
    \item There exists at least one sector at $p\in\Sigma(N)$ if and only if $p$ is not an isolated point of $\Sigma(N)$.
    \item Each sector is an open set because the set $\bigcup_{\gamma\in \Gamma_p}\gamma([0,d_N(p)])$ is closed. This can be proven by using the fact that any sequence of $N$-segments with two endpoints convergent to two distinct points converges to an $N$-segment. The set-valued upper semi-continuity of $\pi_N$ also follows for the same reason.
    \item Each sector is wedged between exactly two $N$-segments in $\Gamma_p$. This is because each pair of elements of $\Gamma_p$ have intersections only at the terminal endpoint, and possibly also at the initial endpoint.
\end{itemize}
\end{remark}

\subsection{Propagation of DC Jordan arcs}

Throughout this subsection, we fix a non-isolated point $p\in\Sigma(N)$ and a strongly convex ball $B_{\delta_0}(p)$ in order to determine a sector at $p$.
Since $B_{\delta_0}(p)$ is strongly convex, each $\gamma\in\Gamma_p$ intersects the circle
$S_{\delta_0}(p):=\{ q\in M \mid d(p,q)=\delta_0\}$ at a unique point $q(\gamma)$.
Let $\textbf{\textup{Sec}}(\gamma_1,\gamma_2)$ denote the sector formed by two $N$-segments $\gamma_1,\gamma_2\in\Gamma_p$ together with the subarc of 
$S_{\delta_0}(p)$ cut off 
by $\gamma_1$ and $\gamma_2$; let $[q_1,q_2]$ denote such a subarc, where $q_i:=q(\gamma_i)$.
Without loss of generality we may assume that all $N$-segments of $\Gamma_p$ have different initial endpoints from each other, by replacing $N$ by $N_c$ in Lemma \ref{lem:closedneighborhood} with $K=\overline{B_{\delta_0}(p)}$ if necessary.
Therefore, by the set-valued upper semicontinuity of $\pi_N$ we can choose sufficiently short subarcs $\tilde N_1\ni q_1$ and $\tilde N_2\ni q_2$ of $[q_1,q_2]$, 
so that the two subsets $N_1:=\pi_N(\tilde N_1)$
and $N_2:=\pi_N(\tilde N_2)$ of $N$ are disjoint.
The following property is particularly important.

\begin{lemma}\label{lem:unique_N_i-segment_endpoints}
For each $i=1,2$, the curve $\gamma_i$ is a unique $N_i$-segment to $p$.
\end{lemma}

\begin{proof}
Let $\alpha$ be any  $N_i$-segment.
Since $\gamma_i$ is an $N_i$-segment, the length of $\alpha$ is equal to that of $\gamma_i.$
Thus,
$\alpha$ is also an $N$-segment emanating from a point of $N_i.$ 
Noting that $\alpha(0)\in N_i=\pi(\tilde{N}_i)$ and there exists at most one $N$-segment emanating from each point of $N$, we see that $\alpha$ must pass through a point of $\tilde{N_i}$.
Since there is no $N$-segment to $p$ passing through  a point of $\textbf{\textup{Sec}}(\gamma_1,\gamma_2),$ $\alpha$ must be  $\gamma_i.$
\end{proof}

We now prove the main propagation result.
We say that a map $f:U\to M^m$ on an open subset $U\subset\R^k$ is locally DC around $p\in U$ (resp.\ strictly differentiable at $p\in U$) if for some (in fact any) local chart $\varphi$ around $f(p)$ the map $\varphi\circ f$ is locally DC in a neighborhood of $p$ (resp.\ strictly differentiable at $p$).
 
\begin{proposition}\label{prop:DCcurve_singular_admit_sector}
    Let $\textbf{\textup{Sec}}(\gamma_1,\gamma_2)$ be any sector at $p\in \Sigma(N).$
    Then there exists a locally DC Jordan arc $c:(-\epsilon_0,\epsilon_0)\to M$ such that
    $c(0)=p$, $c(t)$ is strictly differentiable at $t=0$ with $|c'(0)|\neq0$, $c((-\epsilon_0,\epsilon_0))$ is a one-dimensional DC submanifold, and $c((0,\epsilon_0))\subset\textbf{\textup{Sec}}(\gamma_1,\gamma_2)\cap\Sigma(N).$
    In particular, $c([0,\epsilon_0/2])\subset \Sigma(N)$ is a DC Jordan arc starting from $p$.
\end{proposition}

\begin{proof}
    By Lemma \ref{lem:unique_N_i-segment_endpoints},
    for each $i,$ the function $d_{N_i}$ is differentiable at $p$ and its generalized differential is given by
    \begin{equation*}
    \partial^*d_{N_i}(p)=\{\omega_i\},
    \end{equation*}
    where $\omega_i(\cdot):=g_{v_i}(v_i,\cdot)$ and $v_i:=\gamma'_i(d_N(p)).$
    By Lemmas \ref{lem:well_defined_definition_DC}
    and
    \ref{lem:semiconvexity}, $d_{N_1}$ and $d_{N_2}$ are locally DC on $M.$
    For any sequence $\{\alpha_i:[0,L_i]\to M\}_i$
    of $N$-segments satisfying $\lim_{i\to\infty}\alpha_i(L_i)=p$ and $\alpha_i(L_i)\in\textbf{\textup{Sec}}(\gamma_1,\gamma_2)$ for each $i,$
    any limit $N$-segment of the sequence is an $N$-segment to $p,$ which is either $\gamma_1$ or $\gamma_2.$
    Thus, there exists an open ball $B_{\delta_1}(p),$ where $\delta_1\in(0,\delta_0)$,
    such that $d_N(q)=\min\{d_{N_1}(q),d_{N_2}(q)\}$
    holds for any
    $q\in B_{\delta_1}(p)\cap\textbf{\textup{Sec}}(\gamma_1,\gamma_2)$.
    In particular, if we let $f:=d_{N_1}-d_{N_2}$, then $B_{\delta_1}(p)\cap\textbf{\textup{Sec}}(\gamma_1,\gamma_2)\cap \{f=0\} \subset \textbf{\textup{Sec}}(\gamma_1,\gamma_2) \cap\Sigma(N)$.

    We now show that there is a Jordan arc $c:(-\epsilon_0,\epsilon_0)\to M$ with the desired properties.
    Let $(U,\varphi)$ denote a local chart around $p$ with $\varphi(p)=(0,0)$. 
    Since $df_p\ne 0$, without loss of generality, we may assume that $df_p( \frac{\partial}{ \partial y } )\neq 0$, where $(x(q),y(q))=\varphi(q)$ on $U$.
    Define a smooth function $f_1 : U\to R$ by $f_1\circ \varphi^{-1}(x,y)=x $ on $ \varphi(U)$.
    Let $F:=(f_1\circ \varphi^{-1},f\circ\varphi^{-1})$ on $\varphi(U)$, and $G$ be the local inverse of $F$ around $(0,0)$.
    Then $F$ is DC, and hence so is $G$, cf.\ \cite[Theorem 5.2]{Vesely1989}.
    Let $\overline{c}(t):=G(t,0)$.
    Since $df_p( \frac{\partial}{ \partial y } )\neq 0$, the map $\overline{c}$ is of the form $(t,h(t))$, where $h$ is DC as so is $\bar{c}$.
    Define $c(t):=\varphi^{-1}\circ\overline{c}(t)$ around $t=0$.
    Then $c$ is a DC map, being a Jordan arc mapped into $B_{\delta_0}(p)\cap \{f=0\}$, and $c(0)=p$.
    In addition, $F$ is strictly differentiable at $p$ (by differentiability and semi-concavity of $d_{N_1},d_{N_2}$) and hence so is $G$ at $(0,0)$ so that $c$ is strictly differentiable at $t=0$.
    By the formula $\overline{c}(t)=(t,h(t))$ it is clear that $|c'(0)|\neq0$ and the image of $c$ is a one-dimensional DC submanifold around $t=0$.

    It remains to show that, up to the choice of the parameter orientation, $c(t)\in\textbf{\textup{Sec}}(\gamma_1,\gamma_2)$ for all (small) $t>0$.
    By $df_p=d(d_{N_1}-d_{N_2})_p=\omega_1-\omega_2$ we have
    \begin{align*}
        df_p(v_1)
        =1-\omega_2(v_1)
        >0
        >\omega_1(v_2)-1
        =df_p(v_2).
    \end{align*}
    On the other hand, by using the chain rule and the fact that $f(c(t))\equiv0$, if we let $v:=c'(0)$, then
    $$0=\frac{d}{dt}f(c(t))\Big|_{t=0}=df_p(v)=\omega_1(v)-\omega_2(v),$$
    which means that $v$ or $-v$ bisects the angle of the sector.
    Therefore for each $i$ the curves $\gamma_i$ and $c$ transversally intersect, which implies that, after reversing the parameter of $c$ if necessary, we have $c(t)\in \textbf{\textup{Sec}}(\gamma_1,\gamma_2)\cap B_{\delta_0}(p)\cap \{f=0\} \subset \textbf{\textup{Sec}}(\gamma_1,\gamma_2) \cap \Sigma(N)$ for any small $t>0$.
    (In fact there is no need to assume smallness since $\Sigma(N)$ does not intersect with the interior of $\gamma_i$.)
\end{proof}

\begin{proof}[Proof of Theorem \ref{thm:2D_singular_set}]
  Since the exceptional set $\Sigma(N)\setminus S$ in Theorem \ref{thm:Finslerian_singular_set} (with $m=2$) is an at most countable set, the proof is completed by applying Proposition \ref{prop:DCcurve_singular_admit_sector} to each point of $\Sigma(N)\setminus S$ which admits a nonempty sector, or equivalently which is not isolated in $\Sigma(N)$.
\end{proof}

\begin{remark}
    By Proposition \ref{prop:DCcurve_singular_admit_sector} it follows that
    for any pair $\textbf{\textup{Sec}}(\gamma_1,\gamma_2)$ and $ \textbf{\textup{Sec}}(\beta_1,\beta_2)$  of distinct sectors at $p,$
    there exist DC Jordan arcs $c_{\gamma}:(-\epsilon_0,\epsilon_0)\to M, c_{\beta}:(-\epsilon_1,\epsilon_1)\to M$
    such that $c_\gamma(0)=c_\beta(0)=p,$
    $c_\gamma(-\epsilon_0,0)\subset \textbf{\textup{Sec}}(\gamma_1,\gamma_2)\cap\Sigma(d_N)$ and
    $c_\beta(0,\epsilon_1)\subset \textbf{\textup{Sec}}(\beta_1,\beta_2)\cap\Sigma(d_N).$  
    By making use of these curves, we may define a Jordan arc
    $c:(-\epsilon_0,\epsilon_1)\to M$ by 
    $c(t):=c_\gamma(t)$ for $t\in(-\epsilon_0,0]$ and $c(t):=c_\beta(t)$ for $t\in(0,\epsilon_1).$
    It is clear that $c(0)=p$  and $c( (-\epsilon_0,\epsilon_1) )\subset\Sigma(d_N).$
    In addition, from the bisecting property in the proof of Proposition \ref{prop:DCcurve_singular_admit_sector} we deduce that $c'(0-)$ and $c'(0+)$ are not completely opposite.
    Therefore, when we apply Theorem \ref{thm:DC_IFT} (cf.\ Remark \ref{rem:DC_IFT}), we can choose a coordinate direction in which both $c( (-\epsilon_0,0] )$ and $c( [0,\epsilon_1) )$ are DC graphical, and hence so is the whole $c( (-\epsilon_0,\epsilon_1) )$, cf.\ \cite[Lemma 4.8]{Vesely1989}.
    Thus, for any singular point $p$ of $d_N$ which is not an endpoint (i.e., not admitting a unique sector, cf.\ \cite{Sabau2016}), there exists a one-dimensional DC submanifold in
    $\Sigma(d_N)$ containing $p$ as an interior point.
    This roughly means that a non-isolated singular point $p\in\Sigma(N)$ admits a one-sided DC graphical propagation if and only if $p$ is an endpoint; otherwise two-sided.
\end{remark}

\appendix 

\section{Examples of singular sets of convex functions}

Here we construct some examples of convex functions with special singular sets.
Although the existence of such examples also follows from the general result of Pavlica \cite{Pavlica2004}, here we provide rather simple and direct proofs.

\subsection{Fractional singular set}\label{subsec:counterexample}

In this subsection we construct a concave, or equivalently convex function $u:\mathbf{R}^m\to\mathbf{R}$ whose singular set $\Sigma(u)$ has Hausdorff dimension $\dim_\mathcal{H}\Sigma(u)=m-1-s$, for any given $s\in(0,1)$.

To this end it is sufficient to consider the case of $m=2$.
Indeed, once we obtain such a function $u:\mathbf{R}^2\to\mathbf{R}$ with $\dim_\mathcal{H}\Sigma(u)=1-s$, then the function $\tilde{u}(x^1,x^2,\dots,x^m):=u(x^1,x^2)$ defined on $\mathbf{R}^m$ is also convex and has the property that $\Sigma(\tilde{u})=\Sigma(u)\times\mathbf{R}^{m-2}$ and hence $\dim_\mathcal{H}\Sigma(\tilde{u})=(1-s)+(m-2)=m-1-s$.

Before constructing a concrete example we state a general result on the attainability of the singular set of a convex function.

\begin{proposition}\label{prop:counterexample}
  For any sequence $\{I_j\}_{j\geq1}$ of mutually positively-separated open subintervals of $\R$ such that $I_j\subset[0,1]$, there exists a convex function $u:\mathbf{R}^2\to\mathbf{R}$ such that $\Sigma(u)=C\times\{0\}$, where $C:=[0,1]\setminus\bigcup_{j=1}^\infty I_j$.
\end{proposition}

\begin{proof}[Proof of Proposition \ref{prop:counterexample}]
  Given a sequence $\{I_j\}_{j\geq1}$, we define a sequence $\{v_j\}_{j\geq0}$ of functions in the following way.
  Let $v_0(x,y):=\tfrac{1}{2}x^2+|y|$.
  Let $\phi:\mathbf{R}\to\mathbf{R}$ be a smooth even convex function such that $\phi(y)=|y|$ for $|y|\geq1$; then automatically $\phi>0$ and $\phi(0)\in(0,1)$.
  In addition, writing $I_j=(a_j,b_j)$, we define $r_j:I_j\to\mathbf{R}$ by $r_j(x):=(x-a_j)^2(x-b_j)^2$.
  We then inductively define
  $$v_j(x,y):=v_{j-1}(x,y) + \chi_{I_j}(x)\Big( r_j(x)\phi(\tfrac{y}{r_j(x)})-|y| \Big),$$
  where $\chi_A$ denotes the characteristic function on $A$.
  Notice that $v_j$ is continuous and $v_j\geq v_{j-1}$ holds on $\mathbf{R}^2$ for all $j$.

  We prove that each member $v_j$ is convex by induction.
  For $j=0$ it is trivial.
  Suppose that $v_{j-1}$ is convex and prove that $v_j$ is convex.
  Note first that $v_j=v_{j-1}$ holds on $\mathbf{R}^2\setminus D_j$, where $D_j:=\{(x,y)\mid x\in I_j,\ |y|< r_j(x)\}$, and hence $v_j$ is locally convex on $\mathbf{R}^2\setminus\overline{D_j}$.
  Now it suffices to prove the local convexity of $v_j$ on $E_j:=\overline{D_j}\cap U_j$ with $U_j:=I_j\times\mathbf{R}$; indeed, then $v_j\in C(\mathbf{R}^2)$ is locally convex except at the two points $(\overline{I_j}\setminus I_j)\times\{0\}$ and hence $v_j$ is entirely convex (by approximation of segments).
  For the local convexity on $E_j$ it suffices to show that the Hessian is positive semidefinite on $U_j$, or equivalently that $\mathrm{tr}(D^2v_j)\geq0$ and $\det(D^2v_j)\geq0$ on $U_j$.
  Now the problem is reduced to showing that all $\partial_{xx}v_j$, $\partial_{yy}v_j$, and $\partial_{xx}v_j\partial_{yy}v_j-(\partial_{xy}v_j)^2$ are nonnegative on $U_j$.
  By definition of $v_0,\dots,v_j$ and by the fact that $I_1,\dots,I_j$ are mutually disjoint, the restriction of $v_j$ to $U_j$ is represented by
  $$v_j|_{U_j}(x,y) = \tfrac{1}{2}x^2 + r_j(x)\phi(\tfrac{y}{r_j(x)}).$$
  This implies that $\partial_{yy}v_j|_{U_j}=\tfrac{1}{r_j(x)}\phi''(\tfrac{y}{r_j(x)})\geq0$ since $\phi''\geq0$ and $r_j>0$.
  In addition,
  $$\partial_{xx}v_j|_{U_j}= 1 + \tfrac{y^2r_j'(x)^2}{r_j(x)^3}\phi''(\tfrac{y}{r_j(x)}) + r_j''(x)[\phi(\tfrac{y}{r_j(x)})-\tfrac{y}{r_j(x)}\phi'(\tfrac{y}{r_j(x)})].$$
  The second term is nonnegative, while the third terms is bounded below by $-1$ since $r_j''\geq-(b_j-a_j)^2\geq-1$ and $0\leq \phi(z)-z\phi'(z)\leq 1$ for any $z\in\mathbf{R}$; the last estimate follows since $(\phi(z)-z\phi'(z))'=-z\phi''(z)$ and hence the maximum $\phi(0)\in(0,1)$ is taken at $z=0$ and the minimum $0$ on $|z|\geq1$.
  Therefore, $\partial_{xx}v_j|_{U_j}\geq 0$.
  Finally, we compute
  $$\partial_{xx}v_j\partial_{yy}v_j-(\partial_{xy}v_j)^2 = \tfrac{1}{r_j(x)}\phi''(\tfrac{y}{r_j(x)})\Big( 1+r_j''(x)[ \phi(\tfrac{y}{r_j(x)})-\tfrac{y}{r_j(x)}\phi'(\tfrac{y}{r_j(x)})] \Big)\geq0.$$
  Therefore $v_j$ is also locally convex on $E_j$, and hence $v_j$ is convex on $\mathbf{R}^2$.

  Now we define
  \[
  v_\infty:=\sup_{j\geq0}v_j\ (\leq v_0+1).
  \]
  Then $v_\infty$ is convex since it is the supremum of a sequence of convex functions.
  Let $J:=\bigcup_{j=1}^\infty I_j$.
  We confirm that
  \[
  \Sigma(v_\infty)=(\mathbf{R}\setminus J)\times\{0\}.
  \]
  We first note that this $v_\infty$ is not differentiable on $(\mathbf{R}\setminus J)\times\{0\}$ since if $x_0\not\in J$, then $v_\infty(x_0,y)=v_0(x_0,y)=\frac{1}{2}x_0^2+|y|$ for any $y$ and hence not differentiable at $y=0$.
  In what follows we argue that $v_\infty$ is differentiable outside $(\mathbf{R}\setminus J)\times\{0\}$.
  Since the completement of $(\mathbf{R}\setminus J)\times\{0\}$ is the union of $J\times\R$ and $(\R\setminus J)\times(\R\setminus\{0\})$, it is sufficient to prove differentiability on both the sets.
  Concerning the former set, for any $j$ the function $v_\infty$ is differentiable on the open set $I_j\times\R$ since we have $v_\infty(x,y)=\frac{1}{2}x^2+r_j(x)\phi(\tfrac{y}{r_j(x)})$, and hence taking the union with respect to $j$ implies that $v_\infty$ is differentiable on $J\times\R$.
  Concerning the latter, for any point $p_0=(x_0,y_0)\in(\R\setminus J)\times(\R\setminus\{0\})$, there is $\varepsilon\in(0,|y_0|)$ such that $B_\varepsilon(p_0)\cap \bigcup_{j=1}^\infty D_j=\emptyset$; this follows by the fact that the set $\bigcup_{j=1}^\infty D_j$ is of the form $\{|y|<f(x)\}$ for a nonnegative continuous function $f:\R\to[0,\infty)$ such that $\{f=0\}=\R\setminus J$.
  Then $v_\infty=v_0=\frac{1}{2}x^2\pm y$ holds on $B_\varepsilon(p_0)$,
  where $\pm$ depends on the sign of $y_0$, so that $v_\infty$ is also differentiable at $p_0$.
  This implies the desired differentiability of $v_\infty$.

  Finally, we define a convex function $u:\mathbf{R}^2\to\mathbf{R}$ such that $\Sigma(u)=C\times\{0\}$ with $C:=[0,1]\setminus J$, by using $v_\infty$.
  This is easily done by letting $\rho_0(x):=x^2$ and $\rho_1(x):=(x-1)^2$, and then defining $u$ by
  $$u := v_\infty + \chi_{(-\infty,0)}(x)\big( \rho_0(x)\phi(\tfrac{y}{\rho_0(x)})-|y| \big) + \chi_{(1,\infty)}(x)\big( \rho_1(x)\phi(\tfrac{y}{\rho_1(x)})-|y| \big).$$
  The convexity of $u$ and the fact that
  \[
  \Sigma(u)=C\times\{0\}
  \]
  can be confirmed by parallel (or easier) arguments to the above.
\end{proof}

Now we complete our construction:

\begin{example}[Cantor-like singular set]\label{ex:cantor}
  For $\sigma\in(0,1)$ we define the (standard) generalized Cantor set $C_\sigma\subset[0,1]$ by iteratively deleting (at step $j$) the middle open interval of length $\sigma L_{j-1}$ from each of the remaining $2^{j-1}$ segments of length $L_{j-1}=2^{-(j-1)}(1-\sigma)^{j-1}$.
  It is well known (e.g.\ due to self-similarity) that $\dim_\mathcal{H}C_\sigma=(\log{2})(\log\frac{2}{1-\sigma})^{-1}$.
  Hence, for any given $s\in(0,1)$, by choosing $\sigma:=1-2^{-\frac{s}{1-s}}\in(0,1)$, we can apply Proposition \ref{prop:counterexample} to deduce that there is a convex function $u:\R^2\to\R$ such that $\dim_\mathcal{H}\Sigma(u)=\dim_\mathcal{H}(C_\sigma\times\{0\})=\dim_\mathcal{H}C_\sigma=1-s$.
\end{example}

We finally discuss the reachable gradient $D^*u$ of the convex function $u$ constructed in the proof of Proposition \ref{prop:counterexample}.
After some computations, we deduce that $D^*u(p)=\{(p_1,y)\in\mathbf{R}^2 \mid |y|\leq1\}$ holds for any $p=(p_1,p_2)\in\bd[C]\times\{0\}$, where $\bd[C]$ denotes the topological boundary of the closed set $C$ in $\mathbf{R}$.
In particular, if $C$ has empty interior (like a Cantor set), then $D^*u(p)=\partial u(p)$ for any $p\in\Sigma(u)$.
This implies the fact that no point in $\Sigma(u)$ satisfies Albano--Cannarsa's condition \eqref{eq:AlbanoCannarsa}, which is consistent with the fact that no propagation occurs in a Cantor-like singular set.

\begin{remark}
    There are some results with the following type of assertion: If $\Sigma(N)=\textrm{Max}(N)$ holds, where $\textrm{Max}(N):=\{p\in M\setminus N \mid d_N(p)=\max d_N \}$, then the singular set $\Sigma(N)$ needs to be a submanifold, cf.\ \cite[Theorem 6]{Crasta2016}, \cite[Lemma 5.1]{Innami2019}.
    A natural question is if such a property extends to a more general class including distance functions.
    Our example constructed above shows that a general class of semi-concave functions contains a counterexample.
    In fact, if we define $U(x):=\frac{1}{2}x^2-u(x)$, where $u$ is constructed in the above proof, then $U$ is semi-concave on $\R^2$ but takes the maximum $0$ exactly on the Cantor singular set $\Sigma(u)=C\times\{0\}$, being not a submanifold.
\end{remark}

\subsection{Zigzag singular set}\label{subsec:zigzag}

Here we construct a convex function $u:\R^2\to\R$ that admits a point $p\in\Sigma(u)$ with the following properties: There is a bi-Lipschitz map $f:[-\varepsilon,\varepsilon]\to\Sigma(u)$ with $|f'|\equiv1$ a.e., but there is no Lipschitz graph (hypersurface) $S\subset\R^2$ such that $p\in S\subset\R^2$.
Roughly speaking, this example has a point at which a Lipschitz propagation occurs but a graphical propagation does not.

We begin with a general abstract statement.
Given two points $p,q\in\R^2$ we define $[p,q]:=\{ tp+(1-t)q \mid t\in[0,1]\}$.
Any set of the form $[p,q]$ is called \emph{segment}.
Note that a segment may be a singleton.

\begin{proposition}\label{prop:zigzag}
  Let $\{K_j\}_{j=1}^\infty$ be a sequence of segments in $\R^2$.
  Suppose that $K:=\bigcup_{j=1}^\infty K_j$ is compact.
  Then there exists a convex function $u:\R^2\to\R$ such that $\Sigma(u)=K$.
\end{proposition}

\begin{proof}
  For each $j$ the distance function $u_j:=d(K_j,\cdot)$ defines a nonnegative convex function such that $\Sigma(u_j)=K_j$.
  We prove that the function defined by
  $$u:=\sum_{j=1}^\infty 2^{-j}u_j$$
  gives a desired convex function.
  By the boundedness of $K$ there are $C_1,C_2>0$, depending only on $K$, such that $u_j(p)\leq C_1|p|+C_2$ holds for any $p\in\R^2$ and $j\geq1$.
  In particular, the series $\sum 2^{-j}u_j$ converges locally uniformly, so the limit function $u$ indeed exists and is continuous, and also convex (as it is the supremum of convex functions).
  Below we prove that $\Sigma(u)=K$.

  We first prove that $\Sigma(u)\subset K$, i.e., the limit function $u$ is differentiable on $\R^2\setminus K$.
  Since $K$ is closed, for any $p\in \R^2\setminus K$ there is $r>0$ such that $p\in B_r(p)\subset \R^2\setminus K$.
  For each $j$, since $K_j\subset K$, the restriction of $2^{-j}u_j$ to $B_r(p)$ is of class $C^1$ and satisfies $|\nabla (2^{-j}u_j)|=2^{-j}$ there.
  Hence the series $\sum 2^{-j}u_j$ converges in the $C^1$-topology locally on $\R^2\setminus K$, so that $u$ is of class $C^1$ on $\R^2\setminus K$.

  Finally we prove that $K\subset \Sigma(u)$.
  Fix any $p\in K$.
  Let $j_0$ be the minimal integer such that $p\in K_{j_0}$.
  For notational simplicity we may assume that $p$ is the origin and $K_{j_0}$ lies in the $x$-axis.
  Then we have $u_{j_0}(0,y)=|y|$.
  Since $u$ is of the form $2^{-j_0}u_{j_0}+v_0$, where $v_0:=\sum_{j\neq j_0} 2^{-j}u_j$ is also convex, we have $u(0,y)=2^{-j_0}|y|+v_0(0,y)$.
  By convexity of $v_0(0,\cdot)$ the function $u(0,\cdot)$ is not differentiable at $y=0$.
  Hence $u$ is also not differentiable at the origin $p$, i.e., $p\in\Sigma(u)$.
\end{proof}

A concrete example is then constructed as follows, cf.\ Figure \ref{fig:zigzag}:

\begin{example}[Zigzag singular set]\label{ex:zigzag}
  Let $p_j:=(\frac{1}{2^j},0)\in\R^2$, $q_j:=(\frac{1}{2^j},\frac{1}{2^j})\in\R^2$, $A_j:=[p_{j+1},q_j]$, and $B_j:=[q_j,p_{j+2}]$.
  Take $K_0=\{(0,0)\}$ and $K_{4j+1}:=A_j$, $K_{4j+2}:=RA_j$, $K_{4j+3}:=B_j$, $K_{4j+4}:=RB_j$ for $j\geq0$, where $R$ denotes the horizontal reflection matrix $(x,y)\mapsto(-x,y)$.
  Then the union $K$ of $\{K_j\}_{j\geq0}$ is compact, so that by Proposition \ref{prop:zigzag} there is a convex function $u:\R^2\to\R$ such that $\Sigma(u)=K$.
  The singular set $K$ is a zigzag line of finite length thanks to self-similarity.
  The arclength parameterization of $K$ gives a bi-Lipschitz map from an interval to $K$.
  However $K$ is so zigzag that any Lipschitz graph passing through the origin cannot be contained in $K$.
  Indeed, the sets $K\cap\{\pm x\geq0\}$ can be represented by graphs only in directions $(\cos\theta,\sin\theta)$ such that $\pm\tan\theta\in(\frac{4}{3},2)$, respectively, so the whole $K$ is not graphical around the origin.
  In addition, thanks to Lemma \ref{lem:submanifold_coordinate}, the current example is also not a Lipschitz submanifold in the Euclidean plane $\R^2$ even in the sense of Definition \ref{def:Lipsubmfd_Finsler}.
\end{example}

\begin{center}
    \begin{figure}[htbp]
      \includegraphics[width=200pt]{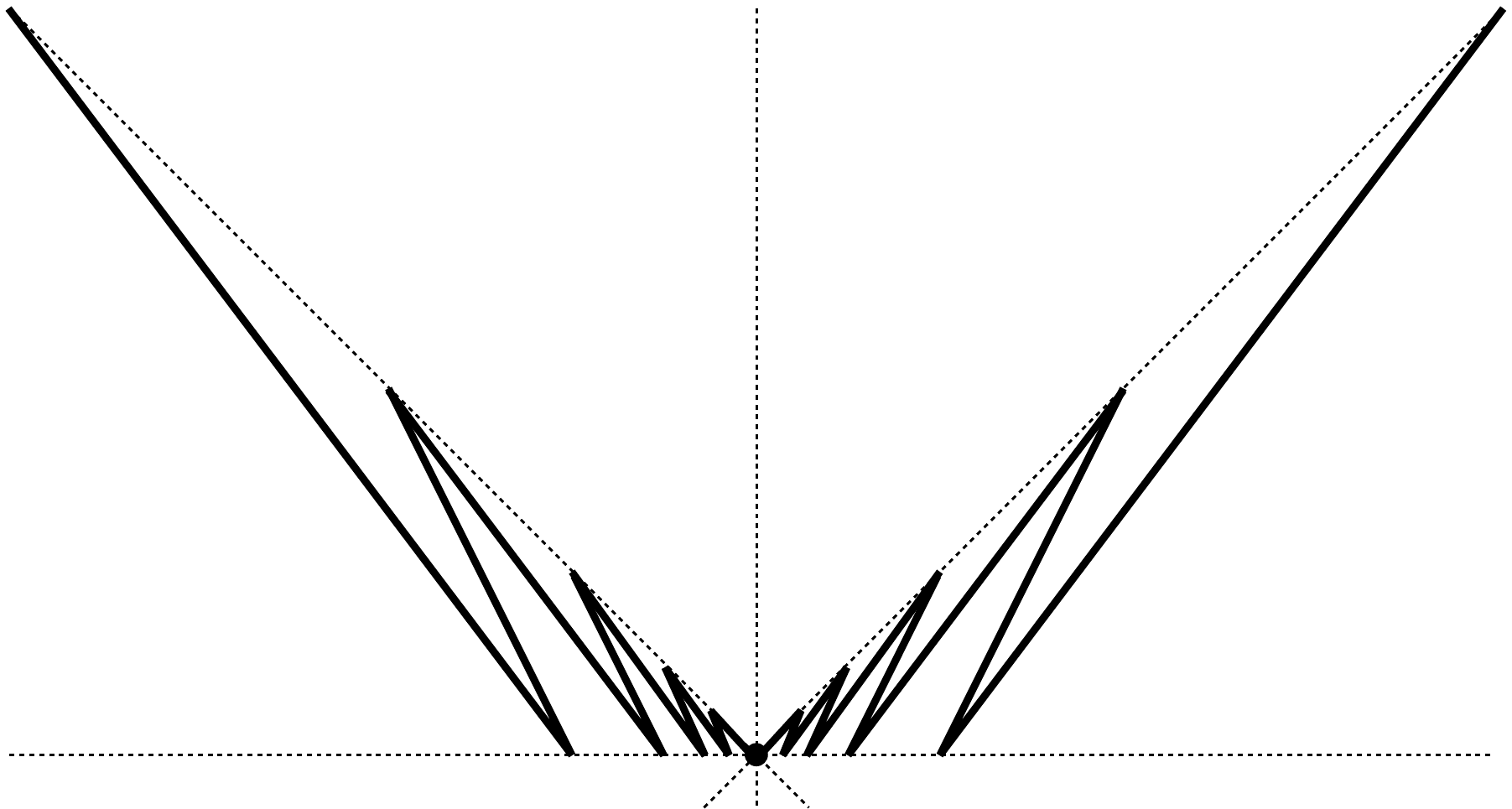}
      \caption{Zigzag singular set.}
      \label{fig:zigzag}
  \end{figure}
\end{center}

Note that the above $u$ does not satisfy Albano--Cannarsa's condition \eqref{eq:AlbanoCannarsa} at the origin.
Hence it seems not yet clear whether the graphical propagation always occurs under condition \eqref{eq:AlbanoCannarsa}:

\begin{problem}\label{prob:semiconcave}
  Let $\Omega\subset\R^m$ be open and $u:\Omega\to\R$ be a semi-concave function.
  Suppose that condition \eqref{eq:AlbanoCannarsa} holds at $x_0\in\Omega$, and also $\dim K_{\partial u(x_0)}(p_0)=m-1$ holds for some $p_0\in\bd[\partial u(x_0)]\setminus D^*u(x_0)$.
  Then, is there a Lipschitz hypersurface $S\subset\R^m$ such that $x_0\in S\subset \Sigma(u)$?
\end{problem}

If not, it seems still interesting to find a suitable subclass of semi-concave functions (e.g.\ solutions to a class of first-order Hamilton--Jacobi equations) for which the above statement holds.

\bibliography{bibliography}

\begin{bibdiv}
\begin{biblist}

\bib{Albano2013}{article}{
      author={Albano, P.},
      author={Cannarsa, P.},
      author={Nguyen, Khai~T.},
      author={Sinestrari, C.},
       title={Singular gradient flow of the distance function and homotopy
  equivalence},
        date={2013},
        ISSN={0025-5831},
     journal={Math. Ann.},
      volume={356},
      number={1},
       pages={23\ndash 43},
         url={https://doi.org/10.1007/s00208-012-0835-8},
      review={\MR{3038120}},
}

\bib{Albano1999}{article}{
      author={Albano, Paolo},
      author={Cannarsa, Piermarco},
       title={Structural properties of singularities of semiconcave functions},
        date={1999},
        ISSN={0391-173X},
     journal={Ann. Scuola Norm. Sup. Pisa Cl. Sci. (4)},
      volume={28},
      number={4},
       pages={719\ndash 740},
      review={\MR{1760538}},
}

\bib{Albano2002}{article}{
      author={Albano, Paolo},
      author={Cannarsa, Piermarco},
       title={Propagation of singularities for solutions of nonlinear first
  order partial differential equations},
        date={2002},
        ISSN={0003-9527},
     journal={Arch. Ration. Mech. Anal.},
      volume={162},
      number={1},
       pages={1\ndash 23},
      review={\MR{1892229}},
}

\bib{Alberti1992}{article}{
      author={Alberti, G.},
      author={Ambrosio, L.},
      author={Cannarsa, P.},
       title={On the singularities of convex functions},
        date={1992},
        ISSN={0025-2611},
     journal={Manuscripta Math.},
      volume={76},
      number={3-4},
       pages={421\ndash 435},
      review={\MR{1185029}},
}

\bib{Alberti1994}{article}{
      author={Alberti, Giovanni},
       title={On the structure of singular sets of convex functions},
        date={1994},
        ISSN={0944-2669},
     journal={Calc. Var. Partial Differential Equations},
      volume={2},
      number={1},
       pages={17\ndash 27},
      review={\MR{1384392}},
}

\bib{Ambrosio1993}{article}{
      author={Ambrosio, L.},
      author={Cannarsa, P.},
      author={Soner, H.~M.},
       title={On the propagation of singularities of semi-convex functions},
        date={1993},
        ISSN={0391-173X},
     journal={Ann. Scuola Norm. Sup. Pisa Cl. Sci. (4)},
      volume={20},
      number={4},
       pages={597\ndash 616},
      review={\MR{1267601}},
}

\bib{Ardoy}{article}{
      author={Ardoy, Pablo~Angulo},
      author={Guijarro, Luis},
       title={Cut and singular loci up to codimension 3},
        date={2011},
     journal={Annales de l'Institut Fourier},
      volume={61},
      number={4},
       pages={1655\ndash 1681},
         url={https://aif.centre-mersenne.org/articles/10.5802/aif.2655/},
      review={\MR{2951748}},
}

\bib{Attali2009}{article}{
      author={Attali, Dominique},
      author={Boissonnat, Jean-Daniel},
      author={Edelsbrunner, Herbert},
       title={Stability and computation of medial axes - a state-of-the-art
  report},
        date={2009},
     journal={In: Farin, G., Hege, H.-C., Hoffman, D., Johnson, C.R., Polthier,
  K. (eds.) Mathematical foundations of scientific visualization, computer
  graphics, and massive data exploration, Springer, Berlin},
       pages={109\ndash 125},
}

\bib{Bao2000}{book}{
      author={Bao, D.},
      author={Chern, S.-S.},
      author={Shen, Z.},
       title={An introduction to {R}iemann-{F}insler geometry},
      series={Graduate Texts in Mathematics},
   publisher={Springer-Verlag, New York},
        date={2000},
      volume={200},
        ISBN={0-387-98948-X},
      review={\MR{1747675}},
}

\bib{Bartke1986}{article}{
      author={Bartke, K.},
      author={Berens, H.},
       title={Eine {B}eschreibung der {N}ichteindeutigkeitsmenge f\"{u}r die
  beste {A}pproximation in der {E}uklidischen {E}bene},
        date={1986},
        ISSN={0021-9045},
     journal={J. Approx. Theory},
      volume={47},
      number={1},
       pages={54\ndash 74},
         url={https://doi.org/10.1016/0021-9045(86)90046-8},
      review={\MR{843455}},
}

\bib{Bialozyt}{article}{
      author={Bia{\l}o\.{z}yt, Adam},
       title={The tangent cone, the dimension and the frontier of the medial
  axis},
        date={2023},
        ISSN={1021-9722,1420-9004},
     journal={NoDEA Nonlinear Differential Equations Appl.},
      volume={30},
      number={2},
       pages={Paper No. 27, 29},
         url={https://doi.org/10.1007/s00030-022-00833-9},
      review={\MR{4530371}},
}

\bib{Bishop2008}{article}{
      author={Bishop, Christopher~J.},
      author={Hakobyan, Hrant},
       title={A central set of dimension 2},
        date={2008},
        ISSN={0002-9939},
     journal={Proc. Amer. Math. Soc.},
      volume={136},
      number={7},
       pages={2453\ndash 2461},
         url={https://doi.org/10.1090/S0002-9939-08-09173-9},
      review={\MR{2390513}},
}

\bib{Blum1967}{incollection}{
      author={Blum, Harry},
       title={{A} {T}ransformation for {E}xtracting {N}ew {D}escriptors of
  {S}hape},
        date={1967},
   booktitle={In: Models for the perception of speech and visual form},
      editor={Wathen-Dunn, Weiant},
   publisher={MIT Press},
     address={Cambridge},
       pages={362\ndash 380},
}

\bib{Caffarelli1979}{article}{
      author={Caffarelli, Luis~A.},
      author={Friedman, Avner},
       title={The free boundary for elastic-plastic torsion problems},
        date={1979},
        ISSN={0002-9947},
     journal={Trans. Amer. Math. Soc.},
      volume={252},
       pages={65\ndash 97},
         url={https://doi.org/10.2307/1998078},
      review={\MR{534111}},
}

\bib{Cannarsa2020}{inproceedings}{
      author={Cannarsa, Piermarco},
      author={Cheng, Wei},
       title={On and beyond propagation of singularities of viscosity
  solutions},
        date={2020},
   booktitle={Proceedings of the {I}nternational {C}onsortium of {C}hinese
  {M}athematicians 2017},
   publisher={Int. Press, Boston, MA},
       pages={141\ndash 157},
      review={\MR{4251110}},
}

\bib{Cannarsa2021}{article}{
      author={Cannarsa, Piermarco},
      author={Cheng, Wei},
       title={Singularities of solutions of {H}amilton-{J}acobi equations},
        date={2021},
        ISSN={1424-9286},
     journal={Milan J. Math.},
      volume={89},
      number={1},
       pages={187\ndash 215},
      review={\MR{4277365}},
}

\bib{Cannarsa2004}{book}{
      author={Cannarsa, Piermarco},
      author={Sinestrari, Carlo},
       title={Semiconcave functions, {H}amilton-{J}acobi equations, and optimal
  control},
      series={Progress in Nonlinear Differential Equations and their
  Applications},
   publisher={Birkh\"{a}user Boston, Inc., Boston, MA},
        date={2004},
      volume={58},
        ISBN={0-8176-4084-3},
      review={\MR{2041617}},
}

\bib{Cannarsa2009}{article}{
      author={Cannarsa, Piermarco},
      author={Yu, Yifeng},
       title={Singular dynamics for semiconcave functions},
        date={2009},
        ISSN={1435-9855},
     journal={J. Eur. Math. Soc. (JEMS)},
      volume={11},
      number={5},
       pages={999\ndash 1024},
      review={\MR{2538498}},
}

\bib{Clarke1990}{book}{
      author={Clarke, F.~H.},
       title={Optimization and nonsmooth analysis},
     edition={Second},
      series={Classics in Applied Mathematics},
   publisher={Society for Industrial and Applied Mathematics (SIAM),
  Philadelphia, PA},
        date={1990},
      volume={5},
        ISBN={0-89871-256-4},
      review={\MR{1058436}},
}

\bib{Crasta2016}{article}{
      author={Crasta, Graziano},
      author={Fragal\`a, Ilaria},
       title={On the characterization of some classes of proximally smooth
  sets},
        date={2016},
        ISSN={1292-8119},
     journal={ESAIM Control Optim. Calc. Var.},
      volume={22},
      number={3},
       pages={710\ndash 727},
         url={https://doi.org/10.1051/cocv/2015022},
      review={\MR{3527940}},
}

\bib{Crasta2019}{article}{
      author={Crasta, Graziano},
      author={Fragal\`a, Ilaria},
       title={Rigidity results for variational infinity ground states},
        date={2019},
        ISSN={0022-2518},
     journal={Indiana Univ. Math. J.},
      volume={68},
      number={2},
       pages={353\ndash 367},
         url={https://doi.org/10.1512/iumj.2019.68.7617},
      review={\MR{3951067}},
}

\bib{Crasta2007}{article}{
      author={Crasta, Graziano},
      author={Malusa, Annalisa},
       title={The distance function from the boundary in a {M}inkowski space},
        date={2007},
        ISSN={0002-9947},
     journal={Trans. Amer. Math. Soc.},
      volume={359},
      number={12},
       pages={5725\ndash 5759},
         url={https://doi.org/10.1090/S0002-9947-07-04260-2},
      review={\MR{2336304}},
}

\bib{Crasta2015}{article}{
      author={Crasta, Graziano},
      author={Malusa, Annalisa},
       title={Existence and uniqueness of solutions for a boundary value
  problem arising from granular matter theory},
        date={2015},
        ISSN={0022-0396},
     journal={J. Differential Equations},
      volume={259},
      number={8},
       pages={3656\ndash 3682},
         url={https://doi.org/10.1016/j.jde.2015.04.032},
      review={\MR{3369258}},
}

\bib{Duda1973}{book}{
      author={Duda, Richard~O},
      author={Hart, Peter~E},
      author={Stork, David~G},
       title={Pattern classification and scene analysis},
   publisher={Wiley New York},
        date={1973},
      volume={3},
}

\bib{Erdoes1945}{article}{
      author={Erd\"{o}s, Paul},
       title={Some remarks on the measurability of certain sets},
        date={1945},
        ISSN={0002-9904},
     journal={Bull. Amer. Math. Soc.},
      volume={51},
       pages={728\ndash 731},
      review={\MR{13776}},
}

\bib{Erdoes1946}{article}{
      author={Erd\"{o}s, Paul},
       title={On the {H}ausdorff dimension of some sets in {E}uclidean space},
        date={1946},
        ISSN={0002-9904},
     journal={Bulletin of the American Mathematical Society},
      volume={52},
       pages={107\ndash 109},
      review={\MR{15144}},
}

\bib{Farin2002}{book}{
      author={Farin, G.},
      author={Hoschek, J.},
      author={Kim, M.S.},
       title={Handbook of computer aided geometric design},
   publisher={Elsevier Science},
        date={2002},
        ISBN={9780444511041},
         url={https://books.google.co.jp/books?id=0SV5G8fgxLoC},
}

\bib{Fremlin1997}{article}{
      author={Fremlin, D.~H.},
       title={Skeletons and central sets},
        date={1997},
        ISSN={0024-6115},
     journal={Proc. London Math. Soc. (3)},
      volume={74},
      number={3},
       pages={701\ndash 720},
         url={https://doi.org/10.1112/S0024611597000233},
      review={\MR{1434446}},
}

\bib{Frerking1989}{article}{
      author={Frerking, J.},
      author={Westphal, U.},
       title={On a property of metric projections onto closed subsets of
  {H}ilbert spaces},
        date={1989},
        ISSN={0002-9939},
     journal={Proc. Amer. Math. Soc.},
      volume={105},
      number={3},
       pages={644\ndash 651},
         url={https://doi.org/10.2307/2046912},
      review={\MR{946636}},
}

\bib{Generau2022}{article}{
      author={G\'{e}n\'{e}rau, Fran\c{c}ois},
      author={Oudet, \'{E}douard},
      author={Velichkov, Bozhidar},
       title={Cut {L}ocus on {C}ompact {M}anifolds and {U}niform
  {S}emiconcavity {E}stimates for a {V}ariational {I}nequality},
        date={2022},
        ISSN={0003-9527},
     journal={Arch. Ration. Mech. Anal.},
      volume={246},
      number={2-3},
       pages={561\ndash 602},
         url={https://doi.org/10.1007/s00205-022-01821-0},
      review={\MR{4514059}},
}

\bib{Hajlasz2022}{article}{
      author={Haj{\l}asz, Piotr},
       title={On an old theorem of {E}rd\"{o}s about ambiguous locus},
        date={2022},
        ISSN={0010-1354},
     journal={Colloq. Math.},
      volume={168},
      number={2},
       pages={249\ndash 256},
         url={https://doi.org/10.4064/cm8460-9-2021},
      review={\MR{4416008}},
}

\bib{Hartman1959}{article}{
      author={Hartman, Philip},
       title={On functions representable as a difference of convex functions},
        date={1959},
        ISSN={0030-8730},
     journal={Pacific J. Math.},
      volume={9},
       pages={707\ndash 713},
         url={http://projecteuclid.org/euclid.pjm/1103039111},
      review={\MR{110773}},
}

\bib{Hartman1964}{article}{
      author={Hartman, Philip},
       title={Geodesic parallel coordinates in the large},
        date={1964},
        ISSN={0002-9327},
     journal={Amer. J. Math.},
      volume={86},
       pages={705\ndash 727},
      review={\MR{173222}},
}

\bib{Hebda1987}{article}{
      author={Hebda, James~J.},
       title={Parallel translation of curvature along geodesics},
        date={1987},
        ISSN={0002-9947},
     journal={Trans. Amer. Math. Soc.},
      volume={299},
      number={2},
       pages={559\ndash 572},
      review={\MR{869221}},
}

\bib{Hebda1994}{article}{
      author={Hebda, James~J.},
       title={Metric structure of cut loci in surfaces and {A}mbrose's
  problem},
        date={1994},
        ISSN={0022-040X},
     journal={J. Differential Geom.},
      volume={40},
      number={3},
       pages={621\ndash 642},
      review={\MR{1305983}},
}

\bib{Innami2019}{article}{
      author={Innami, Nobuhiro},
      author={Itokawa, Yoe},
      author={Nagano, Tetsuya},
      author={Shiohama, Katsuhiro},
       title={Blaschke {F}insler manifolds and actions of projective {R}anders
  changes on cut loci},
        date={2019},
        ISSN={0002-9947},
     journal={Trans. Amer. Math. Soc.},
      volume={371},
      number={10},
       pages={7433\ndash 7450},
         url={https://doi.org/10.1090/tran/7603},
      review={\MR{3939582}},
}

\bib{Itoh1996}{article}{
      author={Itoh, Jin-ichi},
       title={The length of a cut locus on a surface and {A}mbrose's problem},
        date={1996},
        ISSN={0022-040X},
     journal={J. Differential Geom.},
      volume={43},
      number={3},
       pages={642\ndash 651},
      review={\MR{1412679}},
}

\bib{Itoh2001}{article}{
      author={Itoh, Jin-ichi},
      author={Tanaka, Minoru},
       title={A {S}ard theorem for the distance function},
        date={2001},
     journal={Math. Ann.},
      volume={320},
      number={1},
       pages={1\ndash 10 (2001)},
      review={\MR{1835059}},
}

\bib{Kolasinski}{article}{
      author={Kolasi\'{n}ski, S{\l}awomir},
      author={Santilli, Mario},
       title={Regularity of the distance function from arbitrary closed sets},
        date={2023},
        ISSN={0025-5831,1432-1807},
     journal={Math. Ann.},
      volume={386},
      number={1-2},
       pages={735\ndash 777},
         url={https://doi.org/10.1007/s00208-022-02407-7},
      review={\MR{4585161}},
}

\bib{Li2005}{article}{
      author={Li, Yanyan},
      author={Nirenberg, Louis},
       title={The distance function to the boundary, {F}insler geometry, and
  the singular set of viscosity solutions of some {H}amilton-{J}acobi
  equations},
        date={2005},
        ISSN={0010-3640},
     journal={Comm. Pure Appl. Math.},
      volume={58},
      number={1},
       pages={85\ndash 146},
      review={\MR{2094267}},
}

\bib{Lieutier2004}{article}{
      author={Lieutier, A.},
       title={Any open bounded subset of $\mathbb{R}^n$ has the same homotopy
  type as its medial axis},
        date={2004},
     journal={Computer-Aided Design},
      volume={36},
       pages={1029\ndash 1046},
}

\bib{Lions1982}{book}{
      author={Lions, Pierre-Louis},
       title={Generalized solutions of {H}amilton-{J}acobi equations},
      series={Research Notes in Mathematics},
   publisher={Pitman (Advanced Publishing Program), Boston, Mass.-London},
        date={1982},
      volume={69},
        ISBN={0-273-08556-5},
      review={\MR{667669}},
}

\bib{LionsIUMJ}{article}{
      author={Lions, Pierre-Louis},
      author={Seeger, Benjamin},
      author={Souganidis, Panagiotis},
       title={Interpolation results for pathwise {H}amilton-{J}acobi
  equations},
        date={2022},
        ISSN={0022-2518,1943-5258},
     journal={Indiana Univ. Math. J.},
      volume={71},
      number={5},
       pages={2127\ndash 2194},
         url={https://doi.org/10.1512/iumj.2022.71.9174},
      review={\MR{4509830}},
}

\bib{Mantegazza2003}{article}{
      author={Mantegazza, Carlo},
      author={Mennucci, Andrea~Carlo},
       title={Hamilton-{J}acobi equations and distance functions on
  {R}iemannian manifolds},
        date={2003},
        ISSN={0095-4616},
     journal={Appl. Math. Optim.},
      volume={47},
      number={1},
       pages={1\ndash 25},
      review={\MR{1941909}},
}

\bib{Menne2019}{article}{
      author={Menne, Ulrich},
      author={Santilli, Mario},
       title={A geometric second-order-rectifiable stratification for closed
  subsets of {E}uclidean space},
        date={2019},
        ISSN={0391-173X},
     journal={Ann. Sc. Norm. Super. Pisa Cl. Sci. (5)},
      volume={19},
      number={3},
       pages={1185\ndash 1198},
      review={\MR{4012808}},
}

\bib{Miura2016}{article}{
      author={Miura, Tatsuya},
       title={A characterization of cut locus for {$C^1$} hypersurfaces},
        date={2016},
        ISSN={1021-9722},
     journal={NoDEA Nonlinear Differential Equations Appl.},
      volume={23},
      number={6},
       pages={Art. 60, 14},
      review={\MR{3568029}},
}

\bib{Myers1935}{article}{
      author={Myers, Sumner~Byron},
       title={Connections between differential geometry and topology. {I}.
  {S}imply connected surfaces},
        date={1935},
        ISSN={0012-7094},
     journal={Duke Math. J.},
      volume={1},
      number={3},
       pages={376\ndash 391},
      review={\MR{1545884}},
}

\bib{Myers1936}{article}{
      author={Myers, Sumner~Byron},
       title={Connections between differential geometry and topology {II}.
  {C}losed surfaces},
        date={1936},
        ISSN={0012-7094},
     journal={Duke Math. J.},
      volume={2},
      number={1},
       pages={95\ndash 102},
      review={\MR{1545908}},
}

\bib{Ozols1974}{article}{
      author={Ozols, V.},
       title={Cut loci in {R}iemannian manifolds},
        date={1974},
        ISSN={0040-8735},
     journal={Tohoku Math. J. (2)},
      volume={26},
       pages={219\ndash 227},
      review={\MR{390967}},
}

\bib{Pavlica2004}{article}{
      author={Pavlica, David},
       title={On the points of non-differentiability of convex functions},
        date={2004},
        ISSN={0010-2628},
     journal={Comment. Math. Univ. Carolin.},
      volume={45},
      number={4},
       pages={727\ndash 734},
      review={\MR{2103086}},
}

\bib{Petrunin2007}{incollection}{
      author={Petrunin, Anton},
       title={Semiconcave functions in {A}lexandrov's geometry},
        date={2007},
   booktitle={Surveys in differential geometry. {V}ol. {XI}},
      series={Surv. Differ. Geom.},
      volume={11},
   publisher={Int. Press, Somerville, MA},
       pages={137\ndash 201},
         url={https://doi.org/10.4310/SDG.2006.v11.n1.a6},
      review={\MR{2408266}},
}

\bib{Poincare1905}{article}{
      author={Poincar\'{e}, Henri},
       title={Sur les lignes g\'{e}od\'{e}siques des surfaces convexes},
        date={1905},
        ISSN={0002-9947},
     journal={Trans. Amer. Math. Soc.},
      volume={6},
      number={3},
       pages={237\ndash 274},
      review={\MR{1500710}},
}

\bib{Postnikov1967}{book}{
      author={Postnikov, M.~M.},
       title={The {V}ariational {T}heory of {G}eodesics, {T}ranslated by
  {S}cripta {T}echnica, {I}nc.},
      series={Dover {P}hoenix {E}ditions, {R}eprint of the 1967ed.},
   publisher={Dover {P}ublications {I}nc., {M}ineola, {N}ew {Y}ork},
        date={1967},
}

\bib{Rockafellar1970}{book}{
      author={Rockafellar, R.~Tyrrell},
       title={Convex analysis},
      series={Princeton Mathematical Series, No. 28},
   publisher={Princeton University Press, Princeton, N.J.},
        date={1970},
      review={\MR{0274683}},
}

\bib{Sabau2016}{article}{
      author={Sabau, Sorin~V.},
      author={Tanaka, Minoru},
       title={The cut locus and distance function from a closed subset of a
  {F}insler manifold},
        date={2016},
        ISSN={0362-1588},
     journal={Houston J. Math.},
      volume={42},
      number={4},
       pages={1157\ndash 1197},
      review={\MR{3609822}},
}

\bib{Safdari2019}{article}{
      author={Safdari, Mohammad},
       title={The distance function from the boundary of a domain with
  corners},
        date={2019},
        ISSN={0362-546X},
     journal={Nonlinear Anal.},
      volume={181},
       pages={294\ndash 310},
         url={https://doi.org/10.1016/j.na.2019.01.004},
      review={\MR{3901789}},
}

\bib{Santilli2021}{article}{
      author={Santilli, Mario},
       title={Distance functions with dense singular sets},
        date={2021},
        ISSN={0360-5302},
     journal={Comm. Partial Differential Equations},
      volume={46},
      number={7},
       pages={1319\ndash 1325},
      review={\MR{4279967}},
}

\bib{SST2003}{book}{
      author={Shiohama, Katsuhiro},
      author={Shioya, Takashi},
      author={Tanaka, Minoru},
       title={The geometry of total curvature on complete open surfaces},
      series={Cambridge {T}racts in {M}athematics, 159},
   publisher={Cambridge {U}niversity {P}ress, {C}ambridge},
        date={2003},
}

\bib{shiohama1993}{inproceedings}{
      author={Shiohama, Katsuhiro},
      author={Tanaka, Minoru},
       title={The length function of geodesic parallel circles},
        date={1993},
   booktitle={Progress in differential geometry},
      series={Adv. Stud. Pure Math.},
      volume={22},
   publisher={Math. Soc. Japan, Tokyo},
       pages={299\ndash 308},
      review={\MR{1274955}},
}

\bib{Shiohama1996}{incollection}{
      author={Shiohama, Katsuhiro},
      author={Tanaka, Minoru},
       title={Cut loci and distance spheres on {A}lexandrov surfaces},
        date={1996},
   booktitle={Actes de la {T}able {R}onde de {G}\'{e}om\'{e}trie
  {D}iff\'{e}rentielle ({L}uminy, 1992)},
      series={S\'{e}min. Congr.},
      volume={1},
   publisher={Soc. Math. France, Paris},
       pages={531\ndash 559},
      review={\MR{1427770}},
}

\bib{Shiohama2019}{incollection}{
      author={Shiohama, Katsuhiro},
      author={Tiwari, Bankteshwar},
       title={The global study of {R}iemannian-{F}insler geometry},
        date={2019},
   booktitle={Geometry in history},
   publisher={Springer, Cham},
       pages={581\ndash 621},
      review={\MR{3965775}},
}

\bib{Tanaka2020}{article}{
      author={Tanaka, Minoru},
       title={The singular locus of an almost distance function},
        date={2020},
        ISSN={0387-3870},
     journal={Tokyo J. Math.},
      volume={43},
      number={1},
       pages={47\ndash 74},
      review={\MR{4121789}},
}

\bib{Ting1966}{article}{
      author={Ting, Tsuan~Wu},
       title={The ridge of a {J}ordan domain and completely plastic torsion},
        date={1966},
     journal={J. Math. Mech.},
      volume={15},
       pages={15\ndash 47},
      review={\MR{0184503}},
}

\bib{Vesely1989}{article}{
      author={Vesel\'{y}, L.},
      author={Zaj\'{\i}\v{c}ek, L.},
       title={Delta-convex mappings between {B}anach spaces and applications},
        date={1989},
        ISSN={0012-3862},
     journal={Dissertationes Math. (Rozprawy Mat.)},
      volume={289},
       pages={52},
      review={\MR{1016045}},
}

\bib{Vesely1992}{article}{
      author={Vesel\'{y}, Libor},
       title={A connectedness property of maximal monotone operators and its
  application to approximation theory},
        date={1992},
        ISSN={0002-9939},
     journal={Proc. Amer. Math. Soc.},
      volume={115},
      number={3},
       pages={663\ndash 667},
         url={https://doi.org/10.2307/2159212},
      review={\MR{1095227}},
}

\bib{Whitehead1935}{article}{
      author={Whitehead, J. H.~C.},
       title={On the covering of a complete space by the geodesics through a
  point},
        date={1935},
        ISSN={0003-486X},
     journal={Ann. of Math. (2)},
      volume={36},
      number={3},
       pages={679\ndash 704},
      review={\MR{1503245}},
}

\bib{Zajicek1979}{article}{
      author={Zaj\'{\i}\v{c}ek, Lud\v{e}k},
       title={On the differentiation of convex functions in finite and infinite
  dimensional spaces},
        date={1979},
        ISSN={0011-4642},
     journal={Czechoslovak Math. J.},
      volume={29(104)},
      number={3},
       pages={340\ndash 348},
      review={\MR{536060}},
}

\bib{Zajicek1983}{article}{
      author={Zaj\'{\i}\v{c}ek, Lud\v{e}k},
       title={Differentiability of the distance function and points of
  multivaluedness of the metric projection in {B}anach space},
        date={1983},
        ISSN={0011-4642},
     journal={Czechoslovak Math. J.},
      volume={33(108)},
      number={2},
       pages={292\ndash 308},
      review={\MR{699027}},
}

\bib{Zamfirescu1990}{article}{
      author={Zamfirescu, Tudor},
       title={The nearest point mapping is single valued nearly everywhere},
        date={1990},
        ISSN={0003-889X},
     journal={Arch. Math. (Basel)},
      volume={54},
      number={6},
       pages={563\ndash 566},
      review={\MR{1052977}},
}

\end{biblist}
\end{bibdiv}

\end{document}